\documentclass{amsart}

\usepackage{amssymb}
\usepackage{amsthm}
\usepackage{enumitem}

\usepackage{graphicx}
\usepackage{comment}
\usepackage{tikz}
\usepackage{t1enc}
\usepackage{dsfont}
\usepackage{mathrsfs} 
\usepackage{mathtools} 
\usepackage{url}

\newtheorem{thm}{Theorem}
\newtheorem{lem}[thm]{Lemma}
\newtheorem{cor}[thm]{Corollary}

\theoremstyle{definition}
\newtheorem{rem}[thm]{Remark}

\newtheorem{defn}[thm]{Definition}

\newtheorem{fact}[thm]{Fact}
\newtheorem*{ques*}{Question}
\newtheorem*{prob*}{Problem}

\newcommand{\base}{r}
\newcommand{\card}[1]{\left|#1\right|}

\newcommand{\pr}[1]{\left ( #1 \right )}

\newcommand{\cE}{\mathcal{E}}

\newcommand{\cC}{\mathcal{C}}
\newcommand{\cI}{\mathcal{I}}
\newcommand{\cD}{\mathcal{D}}
\newcommand{\Ieps}{{\cI,\epsilon}}
\DeclareMathOperator{\interior}{int}

\newcommand{\sB}{\mathcal{B}}

\newcommand{\A}{\mathscr{A}}

\newcommand{\bs}{\boldsymbol{\Sigma}}
\newcommand{\bp}{\boldsymbol{\Pi}}
\newcommand{\bP}{\boldsymbol{\Pi}}
\newcommand{\bS}{\boldsymbol{\Sigma}}

\newcommand{\bd}{\boldsymbol{\Delta}}

\DeclareMathOperator{\Emp}{\mathcal{E}}

\DeclareMathOperator{\seq}{seq}

\renewcommand{\phi}{\varphi}

\DeclareMathOperator{\M}{\mathcal{M}}
\DeclareMathOperator{\MT}{\mathcal{M}_{\mathit T} }

\DeclareMathOperator{\lang}{\mathscr{L}}

\DeclareMathOperator{\G}{\mathcal{G}}

\newcommand{\dbar}{{\bar d}}

\newcommand{\weakst}{weak$^*$}

\newcommand{\eps}{\varepsilon}
\newcommand{\R}{\mathbb{R}}

\newcommand{\N}{\mathbb{N}}
\newcommand{\Q}{\mathbb{Q}}

\newcommand{\lex}{\text{lex}}
\newcommand{\sm}{\setminus}

\author[D. Airey]{Dylan Airey}
\address[D. Airey]{
Department of Mathematics, University of Texas at Austin, 2515 Speedway, Austin, TX 78712-1202, USA}
\curraddr[D. Airey]{Department of Mathematics, Princeton University, Fine Hall, Washington
Road, Princeton, NJ 08544-1000, USA}
\email{dairey@math.princeton.edu}

\author[S. Jackson]{Steve Jackson}
\address[S. Jackson]{Department of Mathematics, University of North Texas,
General Academics Building 435, 1155 Union Circle,  \#311430, Denton, TX 76203-5017, USA}
\email{stephen.jackson@unt.edu}

\author{Dominik Kwietniak}%
\address[D. Kwietniak]{
Faculty of Mathematics and Computer Science, Jagiellonian University in Krakow, ul. \L o\-jasiewicza 6, 30-348 Krak\'ow, Poland,
\and
Institute of Mathematics, Federal University of Rio de Janeiro, Cidade
Universitaria - Ilha do Fund\~ao, Rio de Janeiro 21945-909, Brazil
}
\email{dominik.kwietniak@uj.edu.pl}
\urladdr{www.im.uj.edu.pl/DominikKwietniak/}

\author[B. Mance]{Bill Mance}
\address[B. Mance]{Institute of Mathematics of Polish Academy of Science,
\'{S}niadeckich 8, 00-656 Warsaw, Poland}
\curraddr[B. Mance]{Uniwersytet im. Adama Mickiewicza w Poznaniu,
  Collegium Mathematicum, ul. Umultowska 87, 61-614 Pozna\'{n}, Poland}
\email{william.mance@amu.edu.pl}

\title[Complexity of normal numbers via generic points]
{Borel complexity of sets of normal numbers via generic points in subshifts with specification}

\begin{document}

\begin{abstract}
We study the Borel complexity of sets of normal numbers in several numeration systems. Taking a dynamical point of view, we offer a unified treatment for continued fraction expansions and base $\base$ expansions,
and their various generalisations: generalised L{\"u}roth series expansions and $\beta$-expansions.
In fact, we consider subshifts over a countable alphabet generated by all possible expansions of numbers in $[0,1)$. Then
normal numbers correspond to generic points of shift-invariant measures.  It turns out that for these subshifts the set of generic points for a shift-invariant probability measure is precisely at the third level of the Borel hierarchy (it is a $\bp^0_3$-complete set, meaning that it is a countable intersection of $F_\sigma$-sets, but it is not possible to write it as a countable union of $G_\delta$-sets). We also solve a problem of Sharkovsky--Sivak on the Borel complexity of the basin of statistical attraction.  The crucial dynamical feature we need is a feeble form of specification. All expansions named above generate subshifts with this property.
Hence the sets of normal numbers under consideration are $\bp^0_3$-complete.
\end{abstract}

\maketitle

\section{Introduction}
Roughly speaking, a \emph{numeration system} assigns to each real number an \emph{expansion}. Here, an expansion is an infinite sequence of \emph{digits} coming from some at most countable set.  A real number is \emph{normal} in a numeration system if all \emph{asymptotic frequencies} of finite blocks of consecutive digits appearing in the expansion are \emph{typical} for the numerations systems. To put some more content into this vague description recall that a real number $\xi$ is normal in base $2$ if in its binary expansion every block of digits of length $k$ appears with asymptotic frequency $1/2^k$. It follows that for every integer $\base\ge 2$ the set of normal numbers in base $\base$ is a first category set of full Lebesgue measure. Also, the normal numbers form a Borel set. As we explain below, the same holds true for all numeration systems we consider. For more on numeration systems, including different views on that theory see \cite{BBLT,Sequences,Rigo}.

Knowing that the sets of normal numbers are Borel it is natural to gauge their complexity using the descriptive hierarchy of Borel sets.
In that hierarchy, the simplest Borel sets are
open ones and their complements (closed sets). On the next level,
there are countable intersections and countable unions of sets at the first level.
These are $G_{\delta}$ and $F_\sigma$ sets, and the third level is
formed by taking countable intersections and unions of sets at the second level. The procedure continues and provides a
stratification of the family of Borel sets into levels corresponding to countable ordinals. It is known that for
an uncountable Polish space
these levels do not collapse: at each level there appear new sets which do
not occur at any lower level of the hierarchy. Thus to every Borel set we can associate its complexity, that is, the lowest
level of the hierarchy at which the set is visible. On the other hand,
determining the position of ``naturally arising'' or ``non-ad hoc''
sets in the hierarchy is a challenging problem. Only a small number of concrete examples are known to
appear only above the third level.

A.~Kechris asked in the
90's whether the set of real numbers that are normal in base two is an
example of a Borel set properly located at the third level, which was
later confirmed by H.~Ki and T.~Linton in \cite{KiLinton}. More precisely, Ki and Linton showed that the set of numbers
that are normal in an integer base $\base\ge 2$ is a $\bp^0_3$-complete set, which means that this set is a countable intersection of $F_\sigma$ sets
and cannot be represented as a countable union of $G_\delta$-sets. Since then many authors have studied
the Borel complexity of various sets related to normal numbers, and
have extended this result in various directions \cite{AireyJacksonManceComplexityNormalPreserves,BecherHeiberSlamanAbsNormal,BecherSlamanNormal,BerosDifferenceSet}.


Here we study analogous problems from the dynamical system
perspective. It allows us to obtain a vast generalization of the Ki
and Linton result. As our primary motivation are applications to
numeration systems we restrict ourselves to symbolic dynamical systems
(\emph{subshifts} for short) and we will address more dynamical
aspects of that theory in a forthcoming paper \cite{AJKM2}.

Before stating our main theorem, let us now briefly explain the
connection between normal numbers and generic points for subshifts.
If $\A$ is a finite or countable\footnote{We call a set \emph{countable} if it has the cardinality $\aleph_0$, where $\aleph_0$ stands for the smallest infinite cardinal. We need this extra
generality to cover continued fractions expansions and some
generalised L\"{u}roth series expansions.} set, which we call the
\emph{alphabet}, then the {\em full shift space} over $\A$ is the pair
$(\A^\omega,\sigma)$ where $\A^\omega$ is endowed with the product
topology induced by the discrete topology on $\A$, and $\sigma$ stands
for the shift map, which is given for $(x_n)_{n\in\omega}\in\A^\omega$
by $\sigma(x)_n=x_{n+1}$. By a {\em subshift} of $\A^\omega$ (or {\em
over $\A$}) we mean a pair $(X,\sigma)$, where $X$ is a nonempty
closed shift-invariant subset of $\A^\omega$, and $\sigma$ is the
shift map restricted to $X$. We also write $\A^n$ for the set all of all \emph{blocks of length $n$ over $\A$}, that is, $\A^n$ stands for the set of all finite sequences $w=w_1\ldots w_n$ with $w_j\in\A$ for $1\le j \le n$ and $n\in\N$.
As we will explain later, the set of sequences of digits which are expansions of real numbers defines a
subshift for each of the numeration systems we consider. Furthermore, normal numbers in these numerations systems always correspond to generic points for some invariant measure of the associated subshift. Recall that a Borel probability measure $\mu$ on $\A^\omega$ is \emph{shift-invariant} if
$\mu(A)=\mu(\sigma^{-1}(A))$ for every Borel set $A\subseteq \A^\omega$. We say a shift-invariant measure $\mu$ is an invariant measure for a subshift $X$ if $X$ contains the support of $\mu$, that is, $\mu(X)=1$. An invariant measure $\mu$ is \emph{ergodic} if for every Borel set $A\subseteq \A^\omega$ the condition $\sigma^{-1}(A)= A$ implies $\mu(A)\in\{0,1\}$
(this is equivalent to saying that if $A\subseteq \sigma^{-1}(A)$ then $\mu(A)\in \{ 0,1\}$).
We say that a finite block $w\in\A^n$ \emph{appears} in $x\in\A^\omega$ at the position $\ell\in \omega$ if $x_{\ell+i-1} = w_i$ for each $1 \leq i \leq n$.  Let $e(w,x,N)$ be the number of times $w$ appears in $x$ at a position $\ell<N$. Let $X$ be a subshift over $\A$ and $\mu$ be an invariant measure. A point $x\in X$ is \emph{generic}
\footnote{This definition differs from the usual definition of a generic point for Polish spaces (cf. \cite[p. 1748]{GK}),
but it is better adapted to the symbolic setting. The equivalence of these two definitions is easy to see
(c.f.\ Corollary~18.3.11 of \cite{Garling}).}
for $\mu$ if
for every finite block $w\in\A^n$ the set of positions at which $w$ appears in $x$ has the frequency equal to the measure of the set of all sequences starting with $w$, that is, if
\[
\lim_{N\to\infty}\frac{e(w,x,N)}{N}=\mu([w]),
\]
where $[w]=\{z\in\A^\omega:z_0=w_1,\ldots,z_{n-1}=w_n\}$.
By the shift-invariance of $\mu$ the measure of $[w]$ is equal to the $\mu$-probability of the occurrence of $w$ at any fixed position $\ell\in\omega$, that is,
\[
\mu([w])=\mu(\{z\in\A^\omega:z_\ell=w_1,\ldots,z_{\ell+n-1}=w_n\}).
\]
The ergodic theorem guarantees  that for every shift-invariant ergodic measure $\mu$ the set of points generic for $\mu$, denoted $G_\mu$, has full measure (this is well-known for compact spaces, for the proof of this fact in the generality considered here, see \cite[Lemma 2.2]{GK}). With this vocabulary the theorem of Ki and Linton becomes the statement that setting $X=\{0,1,\dots, \base-1\}^\omega$, the set of generic points for the Bernoulli measure $\mu$ (which is the product of the countable sequence of uniform probability measures on $\A=\{0,1,\dots, \base-1\}$)
is a $\bp^0_3$-complete set. It is then natural to ask for which subshifts
$(X,T)$ and measures $\mu$  one can prove a similar result about the Borel set $G_\mu\subseteq X$. In particular, we would like to know if the same result holds for other numeration systems than the classical base $\base$-expansions. In terms of the theory of dynamical systems, this amounts to asking for which subshifts and invariant measures the Borel complexity of the set of generic points is a $\bp^0_3$-complete set. Not surprisingly, we are not the first to pose this problem. When the present paper was being finished we learned that in the context of dynamical systems this question was first raised by A.~Sharkovsky and his disciple A.~Sivak (see \cite{SS}, which quotes \cite{Sivak} and \cite{Sharkovsky} as the primary sources, unfortunately these papers are not available in English). Sharkovsky and Sivak worked independently of the normal numbers community and used a slightly different language (for example, they called $G_\mu$ the \emph{basin of attraction of $\mu$}). Sharkovsky and Sivak noted that $G_\mu$ is always a Borel set lying at most at the third level of the hierarchy. It is also easy to see that $G_\mu$ may be empty if $\mu$ is not ergodic. Furthermore, there are easy examples with $G_\mu$ lying at the lowest level of the Borel hierarchy. To see that consider the unit circle $X=\mathbb{R}/\mathbb{Z}$, $\alpha\in\mathbb{R}\setminus\mathbb{Q}$, and let $T$ act as $x\mapsto x+\alpha\bmod 1$. Then for every point $x\in \mathbb{R}/\mathbb{Z}$ its forward $T$-orbit is the sequence $\{n \alpha+x\bmod 1:n\ge 0\}$, so each orbit is uniformly distributed $\bmod\ 1$, which means that every point in the circle is generic for the Lebesgue measure $\lambda$ on $\mathbb{R}/\mathbb{Z}$, so $G_\lambda=\mathbb{R}/\mathbb{Z}$ is a clopen set. The same holds for {\em Sturmian subshifts}, which are symbolic dynamical models for irrational rotations of the circle (see \cite[p. 321]{deVries}). Sharkovsky and Sivak asked if their upper bound for the complexity of $G_\mu$ can be reached (see Problems 3 and 5 in \cite{SS}). As we noted above this asks for a Ki and Linton type result for dynamical systems.
\footnote{Note that the equivalence between normal numbers and generic points for the Bernoulli measure implies that
the Ki and Linton result answers Problem 5 from \cite{SS} in the positive, but does not solve Problem 3 from that paper.} Because of the examples where $G_\mu$ is below the third level
we see that some assumptions on the dynamical systems are required for such a result to hold. It turns out that it suffices to
assume that the system has some form of the {\em specification property}.
The original specification property was introduced by R.\ Bowen in his
paper on Axiom A diffeomorphisms \cite{BowenAxiomA}. The specification property has played an important role in dynamics.
We refer the reader to \cite{KLO} for a discussion of the specification property
and its many variants as well as their  significance in dynamics.
Our main result says that for a subshift $(X,\sigma)$ possessing a feeble form of the specification property the set $G_\mu$ of generic points is $\bp^0_3$-complete for every $\sigma$-invariant Borel probability measure $\mu$.
We also demonstrate that the theorem applies to many dynamical systems generating expansions of real numbers.

Thus the main theorem, which is to our best knowledge the first result of this type for dynamical systems, contains also several previously obtained results on complexity of sets of normal numbers, as well as many new ones.
In particular, we extend the Ki-Linton result to continued fraction expansions, $\beta$-expansions,
and generalized GLS\footnote{Note that GLS stands here for \emph{generalized L\"uroth series}, that is, the notion of a generalized GLS expansion is an extension of the GLS expansion, see \cite{ITN} for more details.} expansions (of which the tent map is a special case).

In addition we note that there are subshifts, which are not so closely connected with numeration systems, but are interesting for the symbolic dynamics community, where our methods apply. These include hereditary subshifts (see Section 4 \cite{KKK} for a more detailed overview).

In \S\ref{sec:voc} we introduce basic definitions and notation, and mention
the overall strategy. We introduce in this section the weak form of the specification property
we require for our main result. In \S\ref{sec:main} we state and prove our main result.
In \S\ref{sec:app} we give a number of applications of the main result including
to continued fractions, $\beta$-expansions, generalized GLS-expansions. The enumeration system corresponding
to the tent map is a special case of a generalized GLS expansion. This then answers a
question of Sharkovsky--Sivak \cite{SS}.

\section{Vocabulary/definitions/notation} \label{sec:voc}
Throughout this paper $\omega=\{0,1,2,\ldots\}$ and $\N=\{1,2,3,\ldots\}$. The cardinality of a finite set $A$ is denoted by $|A|$.
We write $\dbar(A)$ for the {\em upper asymptotic density} of a set $A\subseteq \omega$, that is,
\[
\dbar(A)=\limsup_{n\to\infty}\frac{\card{(A\cap\{0,1,\ldots,n-1\})}}{n}.
\]

\subsection{Borel hierarchy}

We now recall some basic notions from descriptive set theory which gauges the complexity of sets in Polish spaces.
In any topological space $X$, the collection of Borel sets $\sB(X)$ is the smallest
$\sigma$-algebra containing all open sets. Elements of $\sB(X)$ are stratified into levels,
introducing the Borel hierarchy on $\sB(X)$, by defining $\bs^0_1$ to be the family of open sets,
and $\bp^0_1
= \{ X\setminus A\colon A \in \bs^0_1\}$ to be the family of closed sets. For a countable ordinal $\alpha<\omega_1$
we let $\bs^0_\alpha$ be the collection of
countable unions $A=\bigcup_n A_n$ where each $A_n \in \bp^0_{\alpha_n}$
for some ordinal $\alpha_n<\alpha$. We also let $\bp^0_\alpha= \{ X\setminus A\colon A \in \bs^0_\alpha\}$.
Alternatively, $A\in \bp^0_\alpha$ if $A=\bigcap_n A_n$ where
$A_n\in \bs^0_{\alpha_n}$ and $\alpha_n<\alpha$ for each $n$.
We also set $\bd^0_\alpha= \bs^0_\alpha \cap \bp^0_\alpha$ for each countable ordinal $\alpha<\omega_1$, in particular
$\bd^0_1$ is the collection of clopen subsets of $X$.
Note that $\bs^0_2$ is the collection of $F_\sigma$ sets, and $\bp^0_2$
is the collection of $G_\delta$ sets.
For any topological space, $\sB(X)=\bigcup_{\alpha<\omega_1} \bs^0_\alpha=
\bigcup_{\alpha<\omega_1}\bp^0_\alpha$.  It is easy to see that all of the collections
$\bd^0_\alpha$, $\bs^0_\alpha$, $\bp^0_\alpha$ are {\em pointclasses}, that is, they are closed
under inverse images of continuous functions.
Another basic fact is that for any uncountable Polish space $X$, there is no collapse
in the levels of the Borel hierarchy, that is, all the
pointclasses $\bd^0_\alpha$, $\bs^0_\alpha$, $\bp^0_\alpha$, for any ordinal $\alpha <\omega_1$,
are distinct (for a proof, see \cite{Kechris}). Thus, these levels of the Borel hierarchy can be used
to calibrate the descriptive complexity of a set. We say a set
$A\subseteq X$ is $\bs^0_\alpha$ (resp.\ $\bp^0_\alpha$) {\em hard}
if $A \notin \bp^0_\alpha$ (resp.\ $A\notin \bs^0_\alpha$). This says $A$ is
``no simpler'' than a $\bs^0_\alpha$ set. We say $A$ is $\bs^0_\alpha$-{\em complete}
if $A\in \bs^0_\alpha\setminus \bp^0_\alpha$, that is, $A \in \bs^0_\alpha$ and
$A$ is $\bs^0_\alpha$ hard. This says $A$ is exactly at the complexity level
$\bs^0_\alpha$. Likewise, $A$ is $\bp^0_\alpha$-complete if $A\in \bp^0_\alpha\setminus\bs^0_\alpha$.



Let us now discuss our proofs.
In order to determine the exact position of a set $A$ in the Borel hierarchy one must prove
an \emph{upper bound}, that is one must write a condition defining $A$
which shows that it appears at some level in the hierarchy,
and then show a \emph{lower bound}, that is, show
that $A$ does not belong to any lower level in the hierarchy.
To establish a lower bound we use a technique known as ``Wadge
reduction''. It is based on the observation that our hierarchy levels
are all pointclasses 
Thus, for example,  a Borel set $A$ is $\bs^0_\alpha$-hard
if there are a Polish space $Y$, a Borel set $C\subseteq Y$ which is known to be $\bs^0_\alpha$-hard,
and a continuous function $f\colon Y\to X$ such that $f^{-1}(A)=C$. 
The same holds for the $\bp^0_\alpha$
classes. 
Although the whole idea is plain and simple, the difficulty lies in the proper choice of the model space $Y$ and subset $C$, so that
it becomes possible to write down a definition of an appropriate continuous function.

\subsection{Shift spaces} \label{sec:ss}


For a comprehensive introduction to symbolic dynamics we refer to the book \cite{LM} by Lind and Marcus.
For a shift space $X \subseteq \A^\omega$ and integer $n\ge 1$, we
write $\lang_n(X) \subseteq \A^n$ for the set of $n$-blocks
appearing in $X$, that is $w \in \lang_n(X)$ if and only if there
exists some $x \in X$ and $\ell \in \omega$ such that $x_{\ell+i-1} =
w_i$ for all $1 \leq i \leq n$. The \emph{length} of a block $w$ over $\A$ is the number of symbols in $w$ and it is denoted by $|w|$.
We agree that $\A^0$ consists of a single element, called the \emph{empty word}, that is, $\A^0$ contains only the unique block  over $\A$ of length $0$.
By $\A^{<\omega}$ we denote the set of all finite blocks over $\A$ (including the empty word).
We let $\lang(X)=\bigcup_{n\ge 1} \lang_n(X)$ and call $\lang(X)$ the language of $X$.
Note that $\lang(X)$ does not contain the empty word.
For $n\ge 1$ and a block $w\in\A^n$, by
$[w]$ we denote the cylinder consisting of those $x \in
\A^\omega$ with $x_i = w_i$ for $1 \leq i \leq n$. If $X$ is
a subshift and $w\in\lang_n(X)$, then we define $[w]_X=[w]\cap
X$. When there is no ambiguity we drop the dependence on $X$ in our
notation and write just $[w]$ for $[w]_X$.
Henceforth, we enumerate all nonempty blocks in $\A^{<\omega}$ in such a way that all blocks appear before their proper extensions. Any such enumeration induces an analogous enumeration on $\lang(X)$ for every shift space $X\subseteq\A^\omega$, that is we can always write $\lang(X)=\{w_1,w_2,\ldots\}$
in such a way that if $w_i$ is a proper initial segment of $w_j$, then $i<j$. In any such enumeration, we always have $|w_n|\le n$ for every $n\ge 1$.
Note that the whole theory of shift spaces remains the same if instead of $\A^\omega$, we consider $\A^\N$.

\subsection{Frequencies of subblocks}
Recall that $e(w,x,N)$ denotes the number of times a block $w\in\A^{<\omega}$ appears in $x\in\A^\omega$ at a position $\ell<N$.
Similarly, we write $e'(w,u)$ for the number of times $w$ appears as a subblock of $u$. We agree that the empty word {\em never} appears as a subblock of a finite block. We say that a finite block $u$ 
is \emph{$(m,\eps)$-good} for a shift-invariant measure $\mu$ if for every $1\le j \le m$ the fraction of positions at which $w_j$ appears as a subblock of $u$ is $\eps$-close to the $\mu$-measure of the cylinder of $w_j$, that is,
we have
\begin{equation}\label{ineq:me-good}
  \mu([w_j])-\eps < \frac{e'(w_j,u)}{|u|} <\mu([w_j])+\eps\quad\text{for }j=1,\ldots,n.
\end{equation}
(Recall that we have fixed an enumeration of all blocks in $\A^{<\omega}$.)
We say that a sequence $(u_n)_{n\in\N}$ of finite blocks $u_n\in\A^{<\omega}$ with $|u_n|\to\infty$ as $n\to\infty$ \emph{generates} a shift-invariant measure $\mu$ if for every $w\in\A^{<\omega}$ we have
\[
\lim_{n\to\infty}\frac{e'(w,u_n)}{|u_n|-|w|+1}=\mu([w]).
\]
Equivalently, a sequence $(u_n)_{n\in\N}$ of elements of $\A^{<\omega}$ generates a shift-invariant measure $\mu$ if for every $m\in\N$ and $\eps>0$ there is an $n_0$ such that $u_n$ is $(m,\eps)$-good for $\mu$ for every $n\ge n_0$.

For $x\in\A^\omega$, $N\ge 1$, and $w\in\A^k$ we clearly have
\begin{equation}\label{}
e'(w,x_{[0,N)})\le e(w,x,N) \le e'(w,x_{[0,N)})+k-1,
\end{equation}
where $x_{[0,N)}=x_0x_1\ldots x_{N-1}$. It follows that $x\in\A^\omega$ is a generic point for a shift-invariant measure $\mu$ if and only if
the sequence $(x_{[0,N)})_{N\in\N}$ generates $\mu$.

For further reference note that for every $u,v,w\in\A^{<\omega}$ the following holds
\begin{equation}\label{ineq:concat}
e'(w,v)\le e'(w,u)+e'(w,v)\le e'(w,uv) \le e'(w,u)+e'(w,v)+|w|-1.
\end{equation}

\begin{defn}
A sequence $(u_n)_{n\in\N}$ of elements of $\A^{<\omega}$ is
\begin{itemize}
  \item \emph{dominating} if the sequence $(|u_1|+\cdots+|u_n|)/|u_{n+1}|$ converges monotonically to $0$ as $n\to\infty$,
  \item \emph{asymptotically stable} for a shift-invariant measure $\mu$ if for every $\eps>0$ and $m\in\N$ 
  there is $N\in\N$ such that for every $n> N$ there is some $\ell'<|u_n|$ so that $\ell'/|u_{n-1}|<\eps$ and for every $\ell'\le \ell\le|u_n|$ the restriction of $u_n$ to the first $\ell$ letters is $(m,\eps)$-good for $\mu$.
\end{itemize}
\end{defn}
\begin{lem}\label{lem:generating}
If a sequence $(u_n)_{n\in\N}$ of elements $\A^{<\omega}$ is dominating and asymptotically stable for a shift-invariant Borel probability measure $\mu$, then $(u_n)_{n\in\N}$ generates $\mu$ and  the point $x=u_1u_2u_3\ldots$ is generic for $\mu$.
  \end{lem}

\begin{proof} It is clear that $|u_n|\to\infty$ as $n\to\infty$. The definition of asymptotic stability implies immediately that
$(u_n)$ generates $\mu$. Let 
$U_n=u_1u_2\ldots u_n$ for $n\ge 1$.
Applying \eqref{ineq:concat} to $U_n=U_{n-1}u_n$ we have for every $w\in\A^{<\omega}$ that
\[
\frac{e'(w,u_n)}{|u_n|}\frac{|u_n|}{|U_{n}|}\le \frac{e'(w,U_n)}{|U_{n}|} \le 
\frac{|U_{n-1}|}{|u_n|}+\frac{e'(w,u_n)}{|u_n|}
+\frac{|w|-1}{|u_{n}|}.
\]
Taking into account that the sequence $(u_n)$ is dominating, so $|U_{n-1}|/|u_n|$ goes to $0$
and $|U_n|/|u_n|$ converges to $1$ as $n\to\infty$, we have for every $w\in\A^{<\omega}$ that
\[
\lim_{n\to\infty}\frac{e'(w,U_n)}{|U_n|}=\lim_{n\to\infty}\frac{e'(w,u_n)}{|u_n|}=\mu([w]).
\]
It remains to show that $x$ is generic for $\mu$. It is enough to show that for every $m\in\N$ and $\eps>0$ we can find $K>0$ so that $x_{[0,k)}$ is $(m,\eps)$-good for all $k\ge K$.


To this end fix $w\in\A^{<\omega}$ and consider the initial subblock $x_{[0,k)}$ of $x$. It follows that for all sufficiently large
$k$ we can write $x_{[0,k)}=U_n v$ for some $n\in\N$ and a proper subblock $v$ of $U_{n+1}$. Pick $\eps>0$ and $m$ large enough for $w$ to be among $w_1,\ldots,w_m$. Use $m$ and $\eps/2$ to find $N$ as in the definition of asymptotic stability and assume that $k$ is large enough so that the $n$ for which $x_{[0,k)}=U_n v$ is strictly greater than $N$. For that $n$ we can find $\ell'$ as in the definition of asymptotic stability. We have two cases to consider. First, if $|v|<\ell'$, then using \eqref{ineq:concat} we get
\[
e'(w,U_n)\le e'(w,U_n v)\le e'(w,U_n)+ |v|+|w|-1.
\]
It follows that
\begin{equation}\label{ineq:less_ell}
\frac{e'(w,U_n)}{|U_n|}\frac{|U_n|}{|U_n v|}\le \frac{e'(w,U_n v)}{|U_n v|}\le \frac{e'(w,U_n)}{|U_n|}+ \frac{\ell'+|w|-1}{|U_n|}.
\end{equation}
Since $U_n$ is $(m,\eps/2)$-good for $\mu$ we can use \eqref{ineq:less_ell} with \eqref{ineq:me-good} to get
\begin{equation}\label{ineq:less_ell2}
(\mu([w])-\eps/2)\frac{|U_n|}{|U_n v|}\le \frac{e'(w,U_n v)}{|U_n v|}\le \mu([w])+\eps/2+ \frac{\ell'+|w|-1}{|U_n|}.
\end{equation}
Now the left hand side of \eqref{ineq:less_ell2} satisfies
\[(\mu([w])-\eps/2)\frac{|U_n|}{|U_n v|}\ge\mu([w])\left(1-\left(1-\frac{|U_n|}{|U_n v|}\right)\right)-\eps/2\ge \mu([w])-\eps/2-\frac{|v|}{|U_nv|}. \]
Plugging that into \eqref{ineq:less_ell2} we obtain
\begin{equation}\label{ineq:less_ell3}
 \mu([w])-\eps/2-\frac{|v|}{|U_nv|} \le \frac{e'(w,U_n v)}{|U_n v|}\le \mu([w])+\eps/2+ \frac{\ell'+|w|-1}{|U_n|}.
\end{equation}
In the second case $|v|\ge \ell'$, which implies that $v$ is $(m,\eps/2)$-good for $\mu$. By \eqref{ineq:concat} we obtain
\begin{equation}\label{ineq:uv}
e'(w,U_n)+e'(w,v)\le e'(w,U_n v)\le e'(w,U_n)+ e'(w,v)+|w|-1.
\end{equation}
Being $(m,\eps/2)$-good for $\mu$ (see \eqref{ineq:me-good}) means that
\begin{equation}\label{ineq:u_n}
\mu([w])|U_n|-\eps|U_n|/2< e'(w,U_n) < \mu([w])|U_n|+\eps|U_n|/2
\end{equation}
and
\begin{equation}\label{ineq:v}
\mu([w])|v|-\eps|v|/2< e'(w,v) < \mu([w])|v|+\eps|v|/2.
\end{equation}
Applying \eqref{ineq:u_n} and \eqref{ineq:v} to \eqref{ineq:uv} we obtain that
\begin{equation}\label{ineq:more_ell}
\mu([w])-\eps/2<\frac{ e'(w,U_nv)}{|U_n|+|v|} <\mu([w])+\eps/2+\frac{ |w|-1}{|U_n|+|v|}.
\end{equation}
Now, \eqref{ineq:less_ell3} and \eqref{ineq:more_ell} imply that for all sufficiently large $n$ the block $U_nv$ is $(m,\eps)$-good for $\mu$.
\end{proof}



Let $d_H$ stand for the normalised Hamming distance, that is, given two blocks $u=u_1\ldots u_n$ and $w=w_1\ldots w_n$ of equal length we set $d_H(u,w)=|\{1\le j\le n: u_j\neq w_j\}|/n$.

\begin{lem}\label{lem:dbar} Suppose $x,y\in\A^\omega$ and $x\in G_\mu$ for a shift-invariant Borel probability measure $\mu$ on $\A^\omega$.
\begin{enumerate}[label=(\alph*)]
  \item\label{dbar-a} If $\dbar(x,y)=\dbar\left(\{j\in\omega: x_j\neq y_j\}\right)=0$,
then $y\in G_\mu$.
\item\label{dbar-b} If $y=x_{i_0}x_{i_1}x_{i_2}\ldots$ where $(i_j)_{i\in\omega}$ is a strictly increasing sequence of elements of $\omega$ such that
$\dbar\left(\omega\sm\{i_j:j\in\omega\}\right)=0$,
then $y\in G_\mu$ iff $x \in G_\mu$.
\item\label{dbar-c}
Let $x=u_1u_2u_3\ldots$, $y=v_1v_2v_3\ldots$, where $(u_n)_{n\in\N}$ and $(v_n)_{n\in\N}$ are sequences of blocks in $\A^{<\omega}$  satisfying $|u_n|=|v_n|$ for every $n\ge 1$. If
$d_H(u_n,v_n)\to 0$ as $n\to\infty$, and
\begin{equation}
\frac{|u_{n+1}|\cdot d_H(u_{n+1},v_{n+1})}{|u_1|+\ldots+|u_n|}\to 0 \quad\text{as $n\to \infty$},\label{eq:dbar-stab}\end{equation}
then $\dbar(x,y)=\dbar\left(\{j\in\omega: x_j\neq y_j\}\right)=0$.
\item\label{dbar-d} For every $m\in\N$ and $\eps>0$ there exists $\delta>0$ such that if $w\in\A^n$ is $(m,\eps/2)$-good and $w'\in\A^n$ satisfies
$d_H(w,w')<\delta$, then $w'$ is $(m,\eps)$-good.
\end{enumerate}
\end{lem}
\begin{proof}
The first two statements can be found in  \cite{KLO2, WeissBook}. The proof of the fourth
is straightforward. Let us prove the third one.
Fix $\eps>0$. Let $\ell_n=|u_1|+\ldots+|u_n|$ for $n\in\N$. Since $d_H(u_n,v_n)\to 0$ as $n\to \infty$, we find $K$ such that \begin{equation}
d_H(u_n,v_n)<\eps\quad\text{for all $n\ge K$}.
\label{eq:dbar-stab-2}
\end{equation}
Using \eqref{eq:dbar-stab} we obtain $N > K$ such that for every $n\ge N$ we have
\begin{equation}
|u_{n+1}|\cdot d_H(u_{n+1},v_{n+1})<\eps\cdot \ell_n.\label{eq:dbar-stab-3}
\end{equation}
We can pick $N$ so large that $\ell_K/\ell_N<\eps$ holds.
Then for $n> N$ and every $j$ such that $\ell_n\le j<\ell_{n+1}$ it holds that
\begin{multline*}
    \frac{
|\{0\le i<j:x_i\neq y_i\}|}{j}\le\frac{1}{\ell_n}\cdot\left( \sum_{k=1}^{n+1}
|u_k|\cdot d_H(u_k,v_k)\right)\\\le \frac{\sum_{k=1}^{K}
|u_k|}{\ell_n}+\frac{
\sum_{k=K+1}^{n}
|u_{k}|\cdot d_H(u_{k},v_{k})}{\ell_n}+\frac{|u_{n+1}|\cdot d_H(u_{n+1},v_{n+1})}{\ell_n}.
\end{multline*}
Now, note that the first term in the sum above equals $\ell_K/\ell_n$, and so is bounded by $\ell_K/\ell_N<\eps$. By  \eqref{eq:dbar-stab-3}, the third term in the sum above is also bounded by $\eps$. Finally, the same holds for the middle term, because $|u_{K+1}|+\ldots+|u_n|\le \ell_n$ and \eqref{eq:dbar-stab-2}. We have proved that for every $\eps>0$ and all sufficiently large $j$ it holds that
\[
\frac{
|\{0\le i<j:x_i\neq y_i\}|}{j}\le 3\eps,
\] which is what we wanted to show.
\end{proof}

\subsection{Specification for subshifts}
For the general definition of the specification property we refer the reader to \cite{KLO}. We omit it here, as for shift spaces it has a simple combinatorial reformulation. The equivalence of these two definitions is an easy exercise.
\begin{defn}
A shift space $X$ over an at most countable alphabet $\A$
has the \emph{specification property}
if there is a nonnegative integer
$N$ such that if $w_i\in\lang(X)$ for $i=1,\dots,n$ then there are $v_i\in\A^N$ for $i=1,\dots,n-1$
such that $u=w_1v_1w_2v_2\ldots v_{n-1}w_n\in\lang(X)$. Furthermore, we say that $X$ has the \emph{periodic specification property} if, in addition to $v_i\in\A^N$ for $i=1,\dots,n-1$ as above we can take $v_n$ so that the periodic point $x=(w_1v_1w_2v_2\ldots w_nv_n)^\infty$ belongs to $X$.
\end{defn}

Note that if $X$ is a compact subshift, then the specification property and its periodic version are well known to be equivalent. Also, when $X$ is not compact  then the specification property
is no longer necessarily an invariant for topological conjugacy.

The classical specification property is much too strong for our purposes as it does not apply to most $\beta$-shifts.
It is then natural to replace it by a weaker assumption. Looking for such a notion we found out that no existing generalisation of the specification property is fully satisfactory. Therefore we introduce yet another property, which we coin the {\em right feeble specification property}. It is similar to the almost specification property, which was originally defined by
Pfister and Sullivan \cite{PS}, and later modified and renamed by Thompson \cite{Thompson}. The reader
may consult \cite{KLO} for discussion of this property. A variant of the latter property, the {\em right almost specification property}, was considered by Climenhaga and Pavlov (for more details we refer the reader to Definition 2.14 in \cite{CP}). We need a similar condition here to guarantee that the function we will define in the course of our proof of Theorem \ref{thm:fmt} is continuous.

\begin{defn}\label{def:rfs}
We say that a subshift $X$ has the \emph{right feeble specification}
property if there exists a set $\G\subseteq\lang(X)$ satisfying:
\begin{enumerate}
\item\label{cond:rfs_one} a concatenation of words in $\G$ stays in
    $\G$, that is, if $u,v\in\G$, then $uv\in\G$;
  \item\label{cond:rfs_two}
for any $\epsilon >0$ there is an $N=N(\epsilon)$ such that for every
$u\in \G$ and $v\in\lang(X)$ with $|v|\ge N$, there are $s,v'\in\A^{<\omega}$ satisfying $|v'|=|v|$, $0\le |s|\leq \epsilon |v|$,
$d_H(v,v')<\varepsilon$, and $usv' \in \G$.
\end{enumerate}
\end{defn}



It is immediate that the right almost specification property implies the right feeble specification,
in particular the specification property 
implies the feeble specification property (cf.\ \cite[Lemma 2.15]{CP}). It is also easy to see that the weak specification property (see \cite{KLO,KOR}) implies the right feeble specification. We do not know if the weak specification property (or the right feeble specification property) implies the right almost specification property. We suspect that the answer to both questions is ``no'' and an appropriate example can be constructed within the family of subshifts with the weak specification property presented in \cite{KOR}.
\subsection{Irregular set}
Given $w\in\lang(X)$ we define $I_w(X)$ to be the set of all $x\in X$ such that the set of positions at which $w$ appears in $x$ does not have a frequency, that is
\[
\liminf_{N\to\infty}\frac{e(w,x,N)}{N}<\limsup_{N\to\infty}\frac{e(w,x,N)}{N}.
\]
Let $I(X)$ be the {\em irregular set} for $X$, that is, the union of sets $I_w(X)$ over all $w\in\lang(X)$. The {\em quasi-regular set} for $X$ is the complement of $I(X)$, that is, $Q(X)=X\setminus I(X)$.
Both sets are obviously Borel and belong to the third level of the Borel hierarchy.


\section{Main results} \label{sec:main}
\subsection{Subshifts with a feeble specification property}
Theorem~\ref{thm:fmt} below  
applies to 
subshifts
on a countable alphabet satisfying 
a hypothesis  weaker than the (non-periodic) specification property.

Note that we are considering subshifts which are not necessarily compact. 
It forces us to {\em assume} 
that there are at least two shift-invariant Borel probability measures on $X$.
This condition is automatically fulfilled if $X$ is compact with $|X|\geq 2$ (granting the
specification hypothesis).

\begin{thm} \label{thm:fmt} Assume that $\A$ is at most countable and $X$ is a subshift over $\A$ with the right feeble specification property and at least two shift-invariant Borel probability measures.
If $\mu$ is a shift-invariant Borel probability measure on $X$, then every Borel set $B$ satisfying $G_\mu\subseteq B\subseteq Q(X)$, where $G_\mu$ is the set of generic points for $\mu$ and $Q(X)$ is the quasi-regular set, is $\bp_3^0$-hard. In particular, $B$ is
$\bP_3^0$-complete provided that $B$ is a $\bP_3^0$-set. Hence, $G_\mu$ and $Q(X)$ are $\bp^0_3$-complete,
and the irregular set $I(X)$ is $\bs^0_3$-complete.
\end{thm}

\begin{proof} 
First, we note that under our assumptions the set of generic points is nonempty for every shift-invariant Borel probability measure on $X$.
The existence of a generic point for an \emph{arbitrary} shift-invariant measure follows from the
right feeble specification property and Corollary 22 in \cite{KLO2} (formally, the quoted result requires a stronger assumption,
but the proof remains the same when we just assume right feeble specification property).

Fix a shift-invariant Borel probability measure $\mu$ on $X$ and a Borel set $B$ such that $G_\mu\subseteq B\subseteq Q(X)$.
Let $\varepsilon\mapsto N(\varepsilon)$ be the function as implicitly defined for $X$ by Definition~\ref{def:rfs}.
In order to apply Wadge reduction, it suffices to find a Polish metric space $\mathcal{X}$,
a continuous function $\pi\colon \mathcal{X}\to X$ and a  $\bP_3^0$-complete set
$\mathcal{C}_3\subseteq\mathcal{X}$ such that
\[\pi^{-1}(G_\mu)=\pi^{-1}(B)=\pi^{-1}(Q(X))=\mathcal{C}_3 \]
and $\pi^{-1}(I(X))=\mathcal{X}\setminus\mathcal{C}_3$.

We take $\mathcal{X}=\N^\N$ with the topology of pointwise convergence, and choose $\cC_3\subseteq \N^\N$ to be the set of all functions
$\alpha\colon\N\to\N$ attaining any $n\in\N$ only finitely many times, that is,
\[
\mathcal{C}_3=\{\alpha\in\N^\N \colon \liminf_{n\to\infty}\alpha(n)=\infty\}.
\]
It is well-known that $\mathcal{C}_3$ is a $\bP_3^0$-complete set.

In order to define $\pi$ we fix a shift-invariant Borel probability measure $\nu\neq \mu$ on $X$.
Then we fix a $\mu$-generic point $x\in X$ and a $\nu$ generic point $z\in X$.


We will also need auxiliary $\omega$-valued sequences $(a_n)_{n\ge 0}$, $(b_n)_{n\ge 0}$, and $(c_n)_{n\ge 0}$ to be defined in a moment.
It is also convenient to introduce one more auxiliary sequence $(B_n)_{n\ge0}$, so that $B_0=0$ and $B_k=2(b_1+\ldots+b_k)$ for $k\ge 1$.
Given $\alpha\in\N^\N$ and using these sequences we define blocks $u_1,\,u_2,\,\ldots\in\lang(X)$ inductively, defining a group of cardinality $2b_n$ at one step, first for $1\le j \le 2b_1$, by
\begin{equation} \label{eqn:u_1}
u_{j}=\begin{cases}
        x_{[0,a_1)}, & \mbox{if } 0<j\le b_1, \text{ and}\\
        z_{[0,c_1)}, & \mbox{if } b_1<j\le 2b_{1},
      \end{cases}
\end{equation}
and then, assuming that $u_1,\ldots,u_{i}$ have been defined where $i=2(b_1+\cdots+b_n)=B_n$ for some $n\ge 1$, we set
\begin{equation} \label{eqn:u_n}
u_{j}=\begin{cases}
        x_{[0,a_{n+1})}, & \mbox{if } B_n<j\le B_n+b_{n+1},\text{ and} \\
        z_{[0,c_{n+1})}, & \mbox{if }B_n+b_{n+1}<j\le B_n+2b_{n+1}.
      \end{cases}
\end{equation}
We now want to
produce finite blocks $v_0$, $v_1$,  $v_2$,
$v_3,\dots$ in $\lang(X)$ so that all the concatenations $v_0v_1  v_2\cdots  v_n$ for $n\ge 1$ are in $\lang(X)$ and for each $j\ge 1$ the block $v_j$ is close (in an appropriate sense) to $u_j$.

To do so we apply the right feeble specification property inductively.
We start with an arbitrarily chosen $v_0\in\G$. 
Assume that we have defined $v_1,\ldots,v_{j-1}$ for some $j\ge 1$. Then we use the right feeble specification property to obtain $v_j$ so that $v_1v_2\ldots v_{j}\in\G$ and we have $v_{j}=s_{j}u'_{j}$ where $u'_{j}$ has the same length as $u_{j}$, the length of $s_{j}$ is a tiny fraction of  $|u_{j}|$, and the Hamming distance $d_H(u'_{j},u_{j})$ is small.

We will then set $\pi(\alpha)= \sigma^{|v_0|}(v_0v_1v_2\ldots)=v_1 v_2 v_3\cdots$. Note that $\pi(\alpha)\in X$ because $X$ is closed and shift invariant.
With the right choice of the $a_n$'s, $b_n$'s, and $c_n$'s
we will prove that the map $\alpha \mapsto \pi(\alpha)$ is the required reduction.

Now we will define our auxiliary sequences.
For $\alpha \in \N^\N$, let $\alpha'(n)=\min \{n, \alpha(n)\}$.
Let $(a_n)_{n\ge0}$, $(c_n)_{n\ge0}$ 
be sequences of positive integers with $a_0=c_0=1$ growing so fast that for every
$n\in\N$ the following conditions hold:
\stepcounter{equation}
\begin{enumerate} [label=(\theequation\alph*)]
  \item\label{cond:a00} $a_{n}=\alpha'(n)c_n$,
  \item\label{cond:a0} $c_n/n>2^{2n}$,
  \item\label{cond:ai} $c_n> N(1/2^{2n})$,
  \item\label{cond:aii} for each $m\ge c_n$ the block $x_{[0,m)}$ is $(m,1/2^{n+1})$-good for $\mu$, and 
  \item\label{cond:aiii} for each $m\ge c_n$ the block $z_{[0,m)}$ is $(m,1/2^{n+1})$-good for $\nu$.
\end{enumerate}
Now define $b_0=0$ and $(b_n)_{n\ge 1}$ to be a sequence of positive integers satisfying for every
$n\ge 1$ the following conditions:
\stepcounter{equation}
\begin{enumerate} [label=(\theequation\alph*)]
\item\label{cond:b0} $b_n>2^{2n}$,
  \item\label{cond:bi} $a_nb_n> 2^{2n}a_{n+1}$, and
  \item\label{cond:bii} $a_n b_n>  2^{2n} ((a_1+c_1)b_1 +\cdots + (a_{n-1}+c_{n-1})b_{n-1})$.
\end{enumerate}
%


Equations~\eqref{eqn:u_1} and \eqref{eqn:u_n} now define the blocks $u_n$ for $n\ge 1$.
For $n\ge 1$ let
\begin{align*}
\bar{u}'_n  &=(x_{[0,a_{n})})^{b_n}=\underbrace{x_{[0,a_{n})}\ldots x_{[0,a_{n})}}_{b_n\text{ times}}, \text{ and}\\
\bar{u}''_n &=(z_{[0,c_{n})})^{b_n}=\underbrace{z_{[0,c_{n})}\ldots z_{[0,c_{n})}}_{b_n\text{ times}}.
\end{align*}
Note that $\bar{u}'_n$ is the concatenation of the $u_i$'s where $i$ runs from
$B_{n-1}+1$ to $B_{n-1}+b_n$, and $\bar{u}''_n$ is the concatenation of the $u_i$'s where $i$ runs from
$B_{n-1}+b_n+1$ to $B_n=B_{n-1}+2b_n$ for each $n\ge1$. It follows that the points $u_1u_2u_3\ldots\in\A^\omega$ and $\bar{u}'_1\bar{u}''_1\bar{u}'_2\bar{u}''_2\ldots$ are equal.

We claim that:
\begin{enumerate}[label=(\Alph*)]
\item \label{claim:A} If $\alpha\in\mathcal{C}_3$, then $(\bar{u}'_n\bar{u}''_n)$ is dominating and
  an asymptotically stable sequence for $\mu$, which implies by
  Lemma~\ref{lem:generating} that the point $x=u_1u_2u_3\ldots$ is generic for $\mu$.
  \item \label{claim:B} If $\alpha\notin\mathcal{C}_3$, then the sequence $(U'_n)$ given by
  \[
  U'_n=\bar{u}'_1\bar{u}''_1\bar{u}'_2\bar{u}''_2\ldots\bar{u}'_{n-1}\bar{u}''_{n-1}\bar{u}'_n
  \]
  generates $\mu$ and the sequence $(U''_n)$ given by
  \[
  U''_n=\bar{u}'_1\bar{u}''_1\bar{u}'_2\bar{u}''_2\ldots\bar{u}'_n\bar{u}''_n=U'_n\bar{u}''_n
  \]
generates along some subsequence a measure $\nu'$, which is a (nontrivial) convex combination of $\mu$ and $\nu$. Since $\nu'\neq \mu$, we see that $x=u_1u_2u_3\ldots$ is an irregular point.
\end{enumerate}

\begin{proof}[Proof of  Claim \ref{claim:A}]
 Assume that $\alpha\in\mathcal{C}_3$. We first prove the following claim:
\begin{enumerate}[label=(\Alph*')]
  \item \label{claim:Ai} For each $m\in\N$ and $\eps>0$ the first $\ell$ symbols of the block $\bar{u}'_n$ are $(m,\eps)$-good for every sufficiently large $n\ge m$ and $\ell\ge \ell'=2^{n}a_n$.
\end{enumerate}
To see that take any $n\ge m$ and recall that $\bar{u}'_n=(x_{[0,a_n)})^{b_n}$ and $b_n>2^n$ by \ref{cond:b0}. Hence we can consider $2^na_n=\ell'\le \ell < |\bar{u}'_n|=a_nb_n$ and write $\bar{u}'_n$ restricted to the first $\ell$ symbols as a concatenation $\tilde{u}'_n\tilde{u}''_n$ where $\tilde{u}'_n=(x_{[0,a_n)})^{r}$, $r=\lfloor \ell/a_n\rfloor\ge 2^n$ and $|\tilde{u}''_n|<a_n$. We have
for every $1\le j\le m$ that
\[
ra_ne'(w_j,x_{[0,a_n)})\le e'(w_j,\tilde{u}'_n\tilde{u}''_n) \le ra_ne'(w_j,x_{[0,a_n)}) + |\tilde{u}''_n|+|w_j|-1.
\]
Note that $|\tilde{u}''_n|+|w_j|\le (1/2^n)ra_n+m$ and $m/a_n<m/c_n<1/2^{2n}$ by \ref{cond:a00} and \ref{cond:a0}.
Now reasoning as in the proof of Lemma \ref{lem:generating} and using the fact \ref{cond:aii} that $x_{[0,a_n)}$ is $(m,\eps/2)$-good for $\mu$ for all sufficiently large $n$ 
we see that Claim \ref{claim:Ai} holds. To finish the proof of Claim \ref{claim:A} note that
\[|\bar{U}_n|=|\bar{u}'_n\bar{u}''_n|=(a_n+c_n)b_n=(\alpha'(n)+1)c_nb_n.\]
Therefore \ref{cond:bii} implies that $(\bar{U}_n)_{n\in\N}$ is a dominating sequence, and Claim \ref{claim:Ai} together with \ref{cond:bi} and the fact that $\alpha'(n)\to \infty$ as $n\to\infty$ imply that $(\bar{U}_n)_{n\in\N}$ is an asymptotically stable sequence, so we can apply Lemma \ref{lem:generating}.
\renewcommand*{\qedsymbol}{(\(\text{Claim \ref{claim:A} }\blacksquare\))}
\end{proof}

\begin{proof}[Proof of  Claim \ref{claim:B}]
\renewcommand*{\qedsymbol}{(\(\text{Claim \ref{claim:B} }\blacksquare\))}
Observe that Claim \ref{claim:Ai} and \ref{cond:bii} imply that the sequence $(U'_n)_{n\in\N}$, where $U'_n=\bar{u}'_1\bar{u}''_1\bar{u}'_2\bar{u}''_2\ldots\bar{u}'_{n-1}\bar{u}''_{n-1}\bar{u}'_n$, generates $\mu$.
We also have that there exists $M\in\N$ so that $\alpha'(n)=M$ for infinitely many $n$'s. Passing to a subsequence we can assume that this happens for all $n$. The same reasoning as in the proof of Claim \ref{claim:Ai} with \ref{cond:aiii} replacing \ref{cond:aii} yields that the sequence $(U''_n)_{n\in\N}$, where $U''_n=\bar{u}'_1\bar{u}''_1\bar{u}'_2\bar{u}''_2\ldots\bar{u}'_n\bar{u}''_n=U'_n\bar{u}''_n$, generates the measure $(1/(M+1)\nu+M/(M+1)\mu$, which implies that $x=u_1u_2u_3\ldots$ is an irregular point.
\end{proof}
\renewcommand*{\qedsymbol}{\(\square\)}
Unfortunately, we cannot take $\pi(\alpha)=x=u_1u_2u_3\ldots$ because $x$ need not  belong to $X$.
But given $x$ we can use the right feeble specification property to find the sequence of blocks $v_0,v_1,v_2,\ldots$ as outlined above so that $\pi(\alpha)=\sigma^{|v_0|}(v_0v_1v_2\ldots)=v_1v_2v_3\ldots\in X$ and our construction will allow us to use Lemma \ref{lem:dbar} to prove that $\pi(\alpha)$ behaves like $x$.

We start with an arbitrarily chosen $v_0\in\G$. Next we find $v_1$ such that $v_0v_1\in\G$ and $v_1=s_1u'_1$ where $|u'_1|=|u_1|$, $|s_1|<|u_1|/4$, and the Hamming distance $d_H(u_1,u'_1)$ is small (say, $d_H(u_1,u'_1)<1/4$). Note that using Lemma \ref{lem:dbar}\ref{dbar-d} and inequality \eqref{ineq:concat} (and increasing $a_1$ if necessary) we can assume that $v_1\in\G$ is almost as good for $\mu$ as $u_1$. Assume $v_0,v_1, v_2,\dots, v_i$ have been defined for some $i\ge 1$ so that  $v_0v_1v_2\ldots v_i\in\G$ and $B_{n-1}\le i< B_{n-1}+2b_{n}$.
Then the right feeble specification property gives us blocks $s_{i+1}$ and $u'_{i+1}$ such that
\[
|u'_{i+1}|=|u_{i+1}|=\begin{cases}
                   |x_{[0,a_n)}|=a_n, & \mbox{if } i<B_{n-1}+b_n,\text{ and} \\
                   |z_{[0,c_n)}|=c_n, & \mbox{if } B_{n-1}+b_n\le i,
                 \end{cases}
\]
and $|s_{i+1}|\le (1/2^{2n})\cdot |u_i|$ and $d_H(u_{i+1},u'_{i+1})<1/2^{2n}$ (we use here \ref{cond:a00} and \ref{cond:ai}).
We set $v_{i+1}=s_{i+1}u'_{i+1}$ and let
$\pi(\alpha)=\sigma^{|v_0|}(v_0v_1v_2\ldots)=v_1v_2v_3\ldots$. Note that the right feeble specification property guarantees that $v_0v_1v_2\ldots v_iv_{i+1}\in\G$. Define
\[I'=\{j\in\omega:j<|s_1|\text{ or }\exists n\ge 1\text{ and }0<i\le |s_{n+1}|\text{ with }j=|v_1\ldots v_n|+i\}.\]
Since $|s_{n+1}|/|v_{n+1}|$ goes to $0$ as $n\to \infty$ we get that $\dbar( I')=0$.
Therefore Lemma \ref{lem:dbar}\ref{dbar-b} implies that $\pi(\alpha)$ is generic for $\mu$ if and only if $y=u'_1u'_2u'_3\ldots\in G_\mu$.
Putting $x=u_1u_2u_3\ldots$ and $y=u'_1u'_2u'_3\ldots$ into Lemma \ref{lem:dbar}\ref{dbar-c} we see that
$y\in G_\mu$ if and only if $x\in G_\mu$. Note that \eqref{eq:dbar-stab} holds because for
every $n\ge 2$ and $B_{n-1}\le i< B_{n-1}+2b_{n}$ we have
\[
\frac{|u_{i+1}|\cdot d_H(u_{i+1},u'_{i+1}) }{|u_1\ldots u_i|}\le \frac{a_n\cdot (1/2^{2n})}{|U_{n-1}''|}\le \frac{a_{n-1}b_{n-1}\cdot(1/2^{4n-2})}{(a_1+c_1)b_1 +\cdots + (a_{n-1}+c_{n-1})b_{n-1}},
\]
which clearly implies \eqref{eq:dbar-stab} (for the last inequality we have used \ref{cond:bi}
to bound the numerator). 
Similarly, we obtain that
$\pi(\alpha)\in I(X)$ if and only if $y\in I(\A^\omega)$ if and only if $x\in I(\A^\omega)$.
We conclude
$\pi^{-1}(G_\mu)=\pi^{-1}(Q(X))=\mathcal{C}_3$, and $\pi^{-1}(I(X))=\N^\N\setminus\mathcal{C}_3$.
These observations together with Claims \ref{claim:A} and \ref{claim:B} prove that  the map $\alpha \mapsto \pi(\alpha)$ is a reduction map showing
that $B$ is $\bP^0_3$-hard, and so in particular
$G_\mu$ and $Q(X)$ are $\bp^0_3$-complete and $I(X)$ is
$\bs^0_3$-complete, provided the map $\pi$  is continuous. But the continuity is obvious as each initial segment of $\pi(\alpha)$ depends only on $\alpha(1),\ldots,\alpha(n)$ for some $n\in\N$.
\end{proof}
\begin{rem}
Theorem \ref{thm:fmt} holds for any shift-invariant $G_\delta$ subset of $\A^\omega$ with the periodic specification property. The proof requires only a minor modification which we leave for the reader.
\end{rem}

\subsection{Hereditary Shifts} \label{subsec:hereditary}

In \S\ref{sec:app} we present a number of applications of Theorem~\ref{thm:fmt}
to normal numbers defined by using various expansions including $\beta$-expansions,
regular continued fraction expansions, and generalized L\"{u}roth series expansions.
In the remainder of this section we consider a result
which does not follow immediately as a corollary to Theorem~\ref{thm:fmt},
but whose proof uses the same techniques as the one for that theorem.
Namely, we show that the
conclusion of Theorem~\ref{thm:fmt} holds for the class of {\em hereditary shifts}.
Furthermore, we can use Theorem \ref{thm:her} instead of Theorem~\ref{thm:fmt} in the applications presented in \S\ref{sec:app}, because
every subshift considered there is hereditary and has a generic point for each of its invariant measures\footnote{But the proof of the latter fact is anyway based on the specification property.}. Actually, Theorem \ref{thm:her} is valid for an even broader class of subshifts
having a {\em safe symbol} (see \cite{RS}).

 Hereditary subshifts were introduced by Kerr and Li in \cite[p. 882]{KerrLi} (see also \cite{Kwietniak}). The family of hereditary subshifts includes extensively studied classes of subshifts: spacing shifts, $\beta$-shifts, bounded density shifts, $\mathscr{B}$-admissible shifts; also, many examples of $\mathscr{B}$-free shifts. Note also that all full shifts over $\{0,1,\ldots,n\}$ or $\omega$ are hereditary, as well as many sofic shifts and shifts of finite type (golden mean shift for example) (see Section 4 in \cite{KKK} for more details and references).



\begin{defn}
A subshift $X \subseteq \A^\omega$ where $\A=\{0,1,\ldots, n\}$ or $\A=\omega$ is
\emph{hereditary} if $y \leq x$ coordinate-wise and $x \in X$ imply $y
\in X$.
\end{defn}

\begin{defn}
We say that $\gamma \in \A$ is a \emph{safe symbol} for a subshift $X$ over
$\A$ if for every $x\in X$ and $k\ge 0$ we have that the point $y$, where
\[
y_n=\begin{cases}
x_n,&\text{if }n\neq k, \text{ and}\\
\gamma,&\text{if }n=k,
\end{cases}
\]
also belongs to $X$.
\end{defn}

Note that by definition $0$ is a safe symbol for every hereditary subshift, and  a subshift over $\{0,1\}$  is hereditary if and only if $0$ is its safe symbol. It is easy to see examples of subshifts over $\{0,1,2\}$ which are not hereditary but have $0$ as a safe symbol. Shifts with a safe symbol seem to be more important in higher dimensional symbolic dynamics, see \cite{RS} and references therein. Note that we again need to {\em assume}
that there are at least two shift-invariant measures on $X$, as even compact hereditary shifts may have only one invariant measure.

\begin{thm} \label{thm:her}
Assume that $\A$ is at most countable and $X$ is a subshift over $\A$ with a safe symbol (in particular, if $\A=\{0,1,\ldots, n\}$ or $\A=\omega$ and $X$ is a hereditary shift) with more than one shift-invariant Borel probability measure. If $\mu$ is a shift-invariant Borel probability measure on $X$
such that $G_\mu\neq\emptyset$ and $B$ is a Borel set satisfying $G_\mu\subseteq B\subseteq Q(X)$, then $B$ is $\bp_3^0$-hard. In particular, $B$ is
$\bP_3^0$-complete provided that $B$ is a $\bP_3^0$-set. Hence, the set $G_\mu$ is either empty, or is
$\bp^0_3$-complete. In particular, the set $G_\mu$ is $\bp^0_3$-complete for every ergodic measure. Furthermore, $Q(X)$ is a $\bp^0_3$-complete set and $I(X)$ is a $\bs^0_3$-complete set.
\end{thm}

\begin{proof} As in the proof of Theorem~\ref{thm:fmt}, we are going to define a continuous reduction $\pi\colon \N^\N\to X$ with
$\pi^{-1}(G_\mu)=\pi^{-1}(Q(X))=\mathcal{C}_3=\{\beta\in\N^\N \colon \liminf_{n\to\infty}\beta(n)=\infty\}$ and $\pi^{-1}(I(X))=\N^\N\setminus \mathcal{C}_3$.
Without loss of generality we assume that $0$ is a safe symbol for $X$.
By $\delta_\mathbf{0}$ we denote the
Dirac measure concentrated on $0^\infty=000\ldots\in X$. Let $\mu$ be any shift-invariant measure on $X$.
Suppose first that $\mu\neq\delta_\mathbf{0}$, that is, $\mu$ is not
supported on $\{0^\infty\}$. Then $\mu([\gamma] )> 0$ for some $\gamma
\in \{1, \cdots, n-1\}$ by the invariance of $\mu$. Assume that $G_\mu$ is nonempty and take $x
\in G_\mu$. Fix a strictly increasing sequence of nonnegative integers
$(b_n)$ such that $b_0=0$,
\begin{gather}
\lim_{n \to \infty} \frac{b_{n}}{b_{n+1}} = 0,\text{ and}\label{quotient}\\
\lim_{n \to \infty} \frac{\card{\{ k \in [b_n, b_{n+1})  \colon
x_k = \gamma\}}}{b_{n+1} - b_n} = \mu([\gamma]). \label{frequency}
\end{gather}
Fix $\beta\in\N^\N$.
Let $n\in\N$ and let $I_n$ be the set of positions in $[b_{2n-1},b_{2n})$ where $\gamma$ appears in $x$, that is,
\[
I_{n}=\{k\in\N:b_{2n-1}\le k <b_{2n}\text{ and }x_k=\gamma\}.
\]
Let $q_n=|I_n|$. Write $I_n=\{i_1,\ldots,i_{q_n}\}$ where $b_{2n-1}\le i_1<i_2<\ldots<i_{q_n}<b_{2n}$. Let $p_n=q_n-\lfloor q_n/\beta(n)\rfloor+1$ and $J_n=\{i_{p_n},i_{p_n+1},\ldots,i_{q_n}\}$. Note that $q_n/\beta(n)\le |J_n|=\lceil q(n)/\beta(n)\rceil\le q_n/\beta(n)+1$.
Define $\pi \colon \N^\N \to X$ by $\pi(\beta)=y$ where
\[
y_k = \begin{cases}
0, & \text{ if } k\in \bigcup J_n, \text{ and}\\
x_k, & \text{otherwise}. 
\end{cases}
\]
Note that we have defined $y$ so that it agrees with $x$ except on the positions in the set
\[
\bigcup_{n\in\N} J_n\subseteq\bigcup_{n\in\N}[b_{2n-1},b_{2n}).
\]
In particular, for each $n\ge 0$ we have
\begin{equation}\label{equality}
  x_{[b_{2n},b_{2n+1})}=y_{[b_{2n},b_{2n+1})}.
\end{equation}
Note also that for each $n\in\N$ to get $y_{[0,b_{2n})}$ we modify $x_{[0,b_{2n})}$ along at most
\begin{equation}\label{J-set-bound}
b_{2n-1}+\left(\frac{b_{2n} - b_{2n-1}}{\beta(n)}\right)
\end{equation}
positions.
We have $y \in X$ for every $\beta\in\N^\N$ since $y=\pi(\beta)$ is
obtained from $x\in X$ by setting $x_k$ to $0$ for $k\in\bigcup J_n$
and $0$ is the safe symbol for $X$.  The map $\pi$ is continuous since
for each $n\in\N$ it is easy to see that $y_{[0,b_{2n})}$ depends only
  on $x$ and $\beta(1),\ldots,\beta(n)$.

If $\beta \in \mathcal{C}_3$ then $\lim_{n \to \infty} \beta(n) =
\infty$  so the set $\bigcup J_n$ is easily seen to have upper asymptotic density
zero, that is $\dbar(\bigcup J_n)=0$ (use \eqref{quotient} and the bound given by
\eqref{J-set-bound}).
Then we have
\[
\dbar(x,y)=\dbar\left(\{j\in\omega: x_j\neq y_j\}\right)=\dbar(\bigcup J_n)=0.
\]
Using Lemma \ref{lem:dbar}\ref{dbar-a} and the fact that $x\in G_\mu$ we see that $y=\pi(\beta)$ is generic for $\mu$. Hence
$\mathcal{C}_3\subseteq \pi^{-1}(G_\mu)$.

If $\beta \notin
\mathcal{C}_3$ then for some strictly increasing sequence of integers
$(n_k)$ and some $K$ for each $k\in\N$ we have $\beta(n_k) = K <
\infty$. This implies that along the sequence $(2n_k)$ the frequency of the symbol
$\gamma$ in $y_{[b_{2n_k-1}, b_{2n_k})}$ is at most $\mu([\gamma])
  \pr{1 - 1/K} + \varepsilon$ where $\varepsilon$ can be made
  arbitrarily small by choosing $k$ large enough. Thus
\[
\liminf_{k \to \infty} \frac{\card{\{ 0\le s < b_{2n_k} \colon y_s =
  \gamma\}}}{2n_k} < \mu([\gamma]),
\]
while using \eqref{quotient}, \eqref{frequency} and \eqref{equality}
we get
\[
\lim_{k \to \infty} \frac{\card{\{ 0\le s < b_{2k+1} \colon y_s =
  \gamma\}}}{2n_k+1} = \mu([\gamma]).
\]
This implies that if $\beta\notin\mathcal{C}_3$, then $y$ is an
irregular point, $y\in I(X)$. Thus $\pi^{-1}(X\setminus Q(X))=\pi^{-1}(I(X))\supseteq
\N^\N\setminus\mathcal{C}_3$. We conclude
$\pi^{-1}(G_\mu)=\pi^{-1}(Q(X))=\mathcal{C}_3$, and $\pi^{-1}(I(X))=\N^\N\setminus\mathcal{C}_3$. The map $\pi$ is therefore a reduction map proving
that $B$ is $\bP^0_3$-hard and so $G_\mu$ and $Q(X)$ are $\bp^0_3$-complete and $I(X)$ is
$\bs^0_3$-complete.

Now suppose $\mu = \delta_\mathbf{0}$. Let $\nu$ be any ergodic measure on $X$ different from $\mu$ and let $x \in G_\nu$.
Let $\gamma\neq 0$ be any nonzero symbol such that $\nu([\gamma])>0$.
Let $b_n$ be an increasing sequence defined as before with $\mu$ replaced by $\nu$ in \eqref{frequency}.
Then repeat the definition of auxiliary sets $I_n$ and $J_n$ as above,
and define the reduction map $\pi \colon \N^\N \to X$ by $y=\pi(\beta)$ where
\[
y_k = \begin{cases}
x_k, & \text{ if } k\in \bigcup J_n, \text{ and}\\
0, & \text{otherwise}. 
\end{cases}
\]
Reasoning as above we see that $\pi$ is continuous, maps $\mathcal{C}_3$ into $G_\mu\subseteq Q(X)$, and $\N^\N\setminus \mathcal{C}_3$ into $I(X)=X\setminus Q(X)\subseteq X\setminus G_\mu$. This concludes the proof.
\end{proof}

\section{Examples and applications} \label{sec:app}

We present here some rather straightforward but noteworthy consequences of Theorem~\ref{thm:fmt}.
Recall that Ki and Linton \cite{KiLinton} showed that in the classical case of $\base$-ary
expansions the set of normal numbers is $\bP^0_3$-complete. We consider
several classes of generalized expansions for which our theorem provides a similar result.

Consider first the case of generalized GLS expansions (a generalization of ``generalized L\"{u}roth
Series''). These include (generalized) L\"{u}roth series expansions, which in turn include
$\base$-ary expansions, as well as expansions generated by the tent map.
Note that for these applications we can also use Theorem \ref{thm:her} in place of Theorem~\ref{thm:fmt}.
\subsection{Some generalities} Let $\cI=\{I_n=[\ell_n,r_n)\subseteq [0,1]:n\in\cD\}$ be a family of pairwise disjoint intervals indexed by an at most countable set $\cD\subseteq \omega$. We call $\cD$ the \emph{set of digits} of $\cI$. We assume that $\cI$ is a partition of $[0,1]$ modulo sets of zero Lebesgue measure, that is, we assume $\sum_{n\in\cD} (r_n-\ell_n)=1$. We also set $I_\infty=[0,1]\setminus\bigcup_{n\in\cD}I_n$. Note that $1\in I_\infty$ and $I_\infty$ may be uncountable. We also define the \emph{address map} $A_\cI\colon [0,1]\to\mathcal{D}\cup\{\infty\}$ associated with $\cI$ by $A_\mathcal{I}(x)=k$ if and only if $x\in I_k$, where $k\in\mathcal{D}\cup\{\infty\}$. Given \emph{any} (not necessarily continuous) map $T\colon [0,1]\to [0,1]$ such that $T|_{\interior I_n}$ is continuous and strictly monotone for each $n\in\cD$, we define the \emph{itinerary} $\iota(x)$ of $x\in [0,1]$ with respect to $T$ and $\cI$ by $\iota(x)=a_1a_2\ldots\in (\cD\cup\{\infty\})^\N$, where $a_n=A_{\cI}(T^{n-1}(x))$ for $n\ge 1$. Note that $T$ must be Borel measurable.
We say that a Borel probability measure $\mu$ on $[0,1]$ is \emph{$T$-invariant} if $\mu(B)=\mu(T^{-1}(B))$ for every Borel set $B\subseteq[0,1]$. A sequence $(x_n)_{n\ge 0}\subseteq [0,1]$ is \emph{uniformly distributed with respect to $\mu$} if
\[
\lim_{N\to\infty}\frac{1}{N}\left|\{0\le n<N: x_n\in I\}\right|=\mu(I)
\]
for every interval $I\subseteq [0,1]$ with $\mu(\partial I)=0$.
We say that a point $x\in [0,1]$ \emph{generates} $\mu$ if the sequence $(T^n(x))_{n\ge0}$ is uniformly distributed with respect to $\mu$.
\subsection{Generalized GLS expansions} For more details we refer the reader to the book \cite{ITN}.
Let $\cI=\{[\ell_n,r_n):n\in\mathcal{D}\}$ be a family of intervals as above and fix a function $\epsilon\colon \mathcal{D} \to \{0,1\}$.
A \emph{generalized GLS expansion} of $x\in [0,1]$ determined by $(\cI,\epsilon)$ is an element $a_1a_2\ldots$ of $\cD^\N$ such that
\begin{equation}\label{eqn:GLS}
x:=\frac{h(a_1)+\epsilon(a_1)}{s(a_1)}+\sum_{n=2}^\infty (-1)^{\epsilon(a_1)+\ldots+\epsilon(a_{n-1})}\frac{h(a_n)+\epsilon(a_n)}{s(a_1)s(a_2)\ldots s(a_n)},
\end{equation}
where $s(n)=1/(r_n-\ell_n)$ and $h(n)=\ell_n/(r_n-\ell_n)$ for $n\in\mathcal{D}$. Note that for each sequence $a_1a_2\ldots\in\cD^\N$ the formula \eqref{eqn:GLS} defines a point $x\in [0,1]$. We write $\psi_\Ieps$ for the resulting map from $\cD^\N$ into $[0,1]$. Note that $\psi_\Ieps$ is continuous, but not necessarily onto.
Consider the map $T_\Ieps$ such that $T_\Ieps(x)=0$ for each $x\in I_\infty$ and on each interval
$I_n$ we have that $T_\Ieps|_{I_n}$ is a linear function with positive slope from $I_n$ onto $[0,1)$ if $\epsilon(n)=0$, and if $\epsilon(n)=1$ we use the linear map from $I_n$ onto $(0,1]$ with negative slope. This defines a map $T_\Ieps\colon [0,1]\to [0,1]$. 
Let $I_\infty^*$ be the set of all points $x\in [0,1]$ such that the $T_\Ieps$-orbit of $x$ visits $I_\infty$ at some iterate, that is $T_\Ieps^n(x)\in I_\infty$ for some $n\ge 0$.
The itinerary map $\iota_\Ieps$ determines an $(\Ieps)$-GLS expansion for each $x\in [0,1]\setminus I_\infty^*$. The resulting sequences are called \emph{proper} $(\Ieps)$-GLS expansions and are dense in $\cD^\N$.

For each $x$ in the set
\[
\Omega_\Ieps=\bigcap_{k=0}^\infty\bigcup_{n\in\cD}T_\Ieps^{-k}(I_n\cap(0,1))=[0,1]\setminus\bigcup_{k\ge 0}T_\Ieps^{-k}(\{0\})
\]
the itinerary $\iota_\Ieps$ is continuous and gives us the unique $(\Ieps)$-GLS expansion of $x$.
Note that $T_\Ieps^{-1}(\{0\})=[0,1]\setminus\bigcup_{n\in\cD}\interior I_n$ is a closed nowhere dense set, hence $\Omega_\Ieps$ is a dense $G_\delta$ set. Furthermore, the function
$\iota_\Ieps$  is a homeomorphism of $\Omega_\Ieps$ onto the set  $\iota_\Ieps(\Omega_\Ieps)$ with the inverse given by $\psi_\Ieps$ restricted to $\iota_\Ieps(\Omega_\Ieps)$. We also have $\psi_\Ieps|_{\iota_\Ieps(\Omega_\Ieps)}\circ \sigma=T_\Ieps\circ\psi_\Ieps|_{\iota_\Ieps(\Omega_\Ieps)}$.
The \emph{fundamental interval} $\Delta(i_1,\ldots,i_k)$, where $i_1,\ldots,i_k\in\cD\cup\{\infty\}$ is the set
\[
\{x\in [0,1]:\iota_\Ieps(x)\in [i_1\ldots i_k]\subseteq (\cD\cup\{\infty\})^\N\}.
\]
Fix $i_1,\ldots,i_k\in\cD$. Take 
$x\in \Delta(i_1,\ldots,i_k)$.
Writing $p_k/q_k$ for the partial sum of the $(\Ieps)$-GLS expansion for $x$ given by \eqref{eqn:GLS}
(summing up to $n=k$) and setting $\epsilon_j=\epsilon(A_\cI(T_{\Ieps}^{j-1}(x)))$ for $j=1,\ldots,k$ we see that
\[
x=\frac{p_k}{q_k}+(-1)^{\epsilon_1+\ldots+\epsilon_{k}}\frac{T_{\Ieps}^k(x)}{s(i_1)\cdot\ldots\cdot s(i_k)}.
\]
Since $T_{\Ieps}^k(x)$ takes any value in $[0,1)$ if $\epsilon_k=0$ and in $(0,1]$ if $\epsilon_k=1$ we have
\[
\Delta(i_1,\ldots,i_k)=\begin{cases}
                         [d_k,d_k+t_k), & \mbox{if } \epsilon_k=0, \text{ and}\\
                         [d_k-t_k,d_k), & \mbox{otherwise},
                       \end{cases}\]
where $d_k=p_k/q_k$, and $t_k= 1/(s(i_1)\cdot\ldots\cdot s(i_k))$.
\begin{thm} \label{thm:app_gls} Let $T_{\Ieps}$ be the generalized GLS expansion map associated with the pair $(\Ieps)$.
If $\mu$ is a $T_{\Ieps}$-invariant Borel probability measure with $\mu(\{0\})=0$, then the set of $x\in [0,1]$ which generate  $\mu$ is $\bP^0_3$-complete. Furthermore, the set of irregular points for $T_{\mathcal{I},\epsilon}$ is $\bS^0_3$-complete.
\end{thm}
\begin{proof}First note that $\mathcal{D}^\N$ satisfies the assumptions of Theorem \ref{thm:fmt}.
Let $\mu$ be a $T_{\Ieps}$-invariant Borel probability measure on $[0,1]$ such that $\mu(\{0\})=0$.
It follows that $\mu(\Omega_{\Ieps})=1$. Furthermore, no point in $[0,1]\setminus\Omega_\Ieps$
can generate $\mu$, as all these points are eventually mapped to $0$ by $T_\Ieps$. Then we can define
$\nu=\iota_{\Ieps}^*(\mu)$ and $\nu$ is a shift-invariant measure concentrated on $\iota_{\Ieps}(\Omega_\Ieps)\subseteq\cD^\N$.
Since $\nu$ is shift-invariant and its support is contained in $\mathcal{D}^\N$, which has the specification property (c.f.\ \cite{GK,Sigmund}), the set of $\nu$-generic points $G_\nu$ is nonempty and uncountable. On the other hand $\mathcal{D}^\N\setminus \iota_\Ieps(\Omega_\Ieps)$ is at most countable, so $G_\nu\cap\iota_\Ieps(\Omega_\Ieps)\neq\emptyset$. For each $z\in G_\nu\cap\iota_\Ieps(\Omega_\Ieps)$ we have that the $\sigma$-orbit of $z$ visits a cylinder $[a_1\ldots a_k]$ with limiting frequency $\nu([a_1\ldots a_k])$ for every $a_1,\ldots, a_k\in\cD$. Since $z\in \iota_\Ieps(\Omega_\Ieps)$ and
$\psi_{\Ieps}\circ \sigma= T_{\Ieps}\circ \psi_{\Ieps}$ on that set, we have that
$\sigma^n(z)\in [a_1\ldots a_k]$ if and only if
$T_\Ieps^n(\psi_\Ieps(z))\in \Delta(a_1,\ldots,a_k)$.
It follows that $\psi_\Ieps(z)$ visits $\Delta(a_1,\ldots,a_k)$ with frequency $\nu([a_1\ldots a_k])=\iota_{\Ieps}^*(\mu)([a_1\ldots a_k])=
\mu(\iota_{\Ieps}^{-1}([a_1\ldots a_k])=\mu(\Delta(a_1,\ldots,a_k))$.
Furthermore, the boundary points of every basic interval $\Delta(a_1,\ldots,a_k)$ are eventually mapped to $0$, therefore $\mu(\partial\Delta(a_1,\ldots,a_k))=0$. Note also that, for each interval $J$ in $[0,1]$ and $\delta>0$ we can find a countable family $\mathcal{J}$ of disjoint basic intervals contained in $J$ such that $\mu(J\setminus\bigcup\mathcal{J})<\delta$. It follows that $\psi_\Ieps(z)$ generates $\mu$.

Using that $\psi_\Ieps$ is continuous on $\cD^\N$ and that $\psi_\Ieps(z)$ generates $\mu$ if and only if $z\in G_\nu\cap\iota_\Ieps(\Omega_\Ieps)$ we see that to finish the proof we need to show that $G_\nu\cap\iota_\Ieps(\Omega_\Ieps)$ is $\bP^0_3$-complete. But this is obvious since $G_\nu$ is $\bP^0_3$-complete by Theorem \ref{thm:fmt} and $G_\nu\setminus\iota_\Ieps(\Omega_\Ieps)$ is contained in the set of improper expansions, so it is  at most countable.

Now consider any point $x$ which is irregular for $T_\Ieps$, equivalently, with irregular $(\Ieps)$-GLS expansion. Clearly, $x\in\Omega_\Ieps$, hence  the visits of the $T_\Ieps$-orbit of $x$ to some basic interval $\Delta(a_1,\ldots,a_k)$, where $a_1,\ldots,a_k\in\cD$, doesn't have a limiting frequency.
It follows that $z=\iota_\Ieps(x)\in I(\cD^\N)$.
By Theorem~\ref{thm:fmt}, the irregular set $I(\cD^\N)$ of the full shift on $\cD^\N$  is $\bS^0_3$-complete.
Therefore the set of all $x$ irregular for $T_\Ieps$ equals $\psi_\Ieps(I(\cD^\N)\cap\iota_\Ieps(\Omega_\Ieps))$. Reasoning as above we see that the latter must be a $\bS^0_3$-complete set, which ends the proof.
\end{proof}
The Lebesgue measure $\lambda$
on $[0,1]$ is easily seen to be an invariant ergodic measure for $T_\Ieps$ (see \cite[Chapter 3]{ITN}). A real
$x \in [0,1]$ is normal for the $(\Ieps)$-GLS expansion if $x$ generates $\lambda$.

\begin{cor}\label{cor:app_gls}
For any $(\Ieps)$-GLS expansion, the set of numbers which are normal for this expansion is $\bP^0_3$-complete.
\end{cor}

The generalized GLS expansions of Corollary~\ref{cor:app_gls} include several types of expansions
as we record in the following corollary.

\begin{cor}
Each of the following sets is a $\bP^0_3$-complete subset of $[0,1)$: numbers normal for the L\"{u}roth
series expansions (see \cite{ITN}), normal for $Q_\infty$ expansions (see \cite{Qinfinity}),
$\alpha$-L\"{u}roth expansions (see \cite{ITN}), and numbers normal for
$\base$-ary expansions.
\end{cor}



\subsection{$\beta$-expansions}

Our next application concerns $\beta$-expansions. Fix a real number $\beta>1$. Set $
\cD_\beta=\{0,1,\ldots,\lfloor \beta \rfloor\}$. For $n\in\cD_\beta$ let
\[
I_n=\begin{cases}
      \left[{n}/{\beta},{(n+1)}/{\beta}\right), & \mbox{if } 0\le n< \lfloor \beta \rfloor, \text{ and}\\
      \left[{\lfloor \beta \rfloor}/{\beta},1\right), & \mbox{otherwise}.
    \end{cases}
\]
Define $\cI_\beta=\{I_n\colon n\in\cD_\beta \}$ and $T_\beta(x)=\beta x\bmod 1$ for $x\in[0,1]$.
A \emph{$\beta$-expansion} of $x \in [0,1]$ is a sequence $d_1d_2\ldots\in \cD_\beta^\N$ so that
$x= \sum_{i=1}^\infty \frac{d_{i}}{\beta^i}$. For each $x \in [0,1)$ the itinerary of $x$ with respect to $T_\beta$ and $\cI_\beta$, denoted by $\iota_\beta(x)=d_1d_2\ldots$ and given by the formula $d_i= \lfloor\beta T^{i-1}_\beta(x)\rfloor$ for $i\ge 1$, defines a sequence $\vec{d}=d_1d_2\ldots\in
\cD_\beta^\N$ which is a $\beta$-expansion of $x$. We use the same formula to define  the $\beta$-expansion of
$1$, denoted by $1_\beta$. We let $\vec{e}=e_1e_2\ldots=1_\beta$ if $1_\beta$ does not end in a tail of $0$'s,
and if $1_\beta=d_1\ldots d_k00\ldots$, where $d_k\neq 0$, then we let
$\vec{e}=e_1e_2\ldots$ be the periodic sequence $(d_1\dots d_{k-1} d_k-1)^\infty$.
We say a sequence of digits $\vec{d}=d_1d_2\ldots\in \cD_\beta^\N$ is a {\em proper $\beta$-expansion} if there exists $x\in [0,1)$ such that $\vec{d}=\iota_\beta(x)$.
A point $x\in (0,1)$ has a unique $\beta$-expansion $\vec{d}$ given by $\vec{d}=\iota_\beta(x)$ if and only if $T_\beta^{i}(x)\neq 0$ for each $i\in\N$.
If $x\in (0,1)$ and $T_\beta^{i}(x)=0$ for some $i\in\N$, then $x$ has exactly two $\beta$-expansions: one proper, and the other we call \emph{improper}.
Clearly, the set of improper $\beta$-expansions is countable.

A sequence $d_1d_2\ldots\in \cD_\beta^\N$ is \emph{admissible} if it is a $\beta$-expansion of some $x\in[0,1]$. We recall
the following well-known fact (\cite{Parry}): a sequence $d_1d_2\ldots\in \cD_\beta^\N$
is a proper $\beta$-expansion if and only if for all $i\in\N$ we have that $d_id_{i+1}\ldots <_\lex e_1e_2\dots$,
where $<_\lex$ denotes the strict lexicographic ordering on $\cD_\beta^\N$. We note that the sequence
$\vec{e}$ itself also has the property that for any shift $\sigma^k(\vec{e})=e_ke_{k+1}\ldots$ we have
$\sigma^k(\vec{e})\leq_\lex \vec{e}$. Observe that if $\vec{d}$ is an admissible sequence and $\vec{d}'$
is obtained from $\vec{d}$ by lowering certain digits, then $\vec{d}'$
is also admissible. The set of proper $\beta$-expansions $Y_\beta:=\iota_\beta([0,1))\subseteq \cD_\beta^\N$
is shift-invariant but not closed in $\cD_\beta^\N$. Let $X_\beta$ be the closure of $Y_\beta$ in $\cD_\beta^\N$,
so $X_\beta$ is a subshift of $\cD_\beta^\N$ known as a \emph{$\beta$-shift}. Every $\beta$-shift is hereditary.
We can characterise  $X_\beta$ as the set of admissible sequences, or equivalently, those sequences
$d_1d_2\ldots\in \cD_\beta^\N$ such that for all $i\in\N$ we have that
$d_id_{i+1}\ldots \leq_\lex \vec{e}$. From this it follows that the set of improper $\beta$-expansions $X_\beta\setminus Y_\beta$
is countable and $Y_\beta$ is a dense $G_\delta$ subset of $X_\beta$.
There is a continuous map $\psi_\beta\colon X_\beta\to [0,1]$ given by $\psi_\beta(d_1d_2\ldots)=\sum_{i=1}^\infty \frac{d_{i}}{\beta^i}$.
The restriction of the map $\psi_\beta$ to $Y_\beta$ is a bijection onto $[0,1]$, but its inverse $\iota_\beta=(\psi_\beta|_{Y_\beta})^{-1}$ is not continuous on $[0,1]$. But $\iota_\beta$ is continuous on a subset $\Omega_\beta$ of $[0,1]$ defined by
\[
\Omega_\beta=[0,1)\setminus
\bigcup_{k\ge 0}T_\beta^{-k}(\{0\}).
\]
Note that every point in $\Omega_\beta$ has a unique $\beta$-expansion, $0\notin \Omega_\beta$, but $0$ has a unique $\beta$-expansion, and the only other point in $[0,1]$ which may have a unique $\beta$-expansion and stay outside of $\Omega_\beta$ is $1$. Let $Z_\beta$ be the set of unique $\beta$-expansions of points in $(0,1)$. We have $Z_\beta=\iota_\beta(\Omega_\beta)\subseteq Y_\beta$, more precisely
\[Z_\beta=
X_\beta\setminus\left(
\{d_1d_2\ldots\in\cD_\beta^\N:(\exists\,i\ge 1)\,d_id_{i+1}\ldots=0^\infty\text{ or }d_id_{i+1}\ldots=\vec{e}\}
\right).
\]
Thus $\iota_\beta$ restricted to $\Omega_\beta$ is the inverse of $\psi_\beta|_{Z_\beta}$.
The admissible sequences can also be described as follows. Let $G$ be the labelled directed graph
(with \emph{loops}, that is edges whose initial and terminal vertices are the same) on the vertex set $\omega$ 
defined as follows. Each vertex $i\in \omega$ is the initial vertex of the
edge leading to the  vertex $i+1$, and labelled with $e_{i+1}$. If $e_{i+1}>0$,
then we add $e_i$ many edges from $i$ to $0$, and label them with $0,1,\dots,e_i-1$.
The elements of $X_\beta$ are obtained by taking an infinite path starting at $0$ and reading off the sequence of labels of
the edges used to construct the path. The proper $\beta$-expansions (elements of $Y_\beta$) are exactly the infinite sequences of labels of paths obtained by starting at
the vertex $0$ and  returning to $0$ infinitely many times. In particular, $\lang(X_\beta)$ corresponds to the labels of finite paths through $G$
starting at $0$.
Note that as $e_1=\lfloor\beta\rfloor$, there are $e_1>0$ edges from $0$ to $0$
(and these are the only loops in the graph $G$).

\begin{lem}\label{lem:rfs-for-beta}
For every $\beta>1$ the $\beta$-shift $X_\beta$ has the right feeble specification property.
\end{lem}
\begin{proof}
Let $\G\subseteq \lang(X_\beta)$ be the set of $w \in \cD_\beta^{< \omega}$ corresponding to closed paths
in the graph $G$ which start and end at the vertex $0$. Clearly if $u \in \G$
and $v \in \lang(X_\beta)$ then $uv \in \lang(X_\beta)$ (since $v$ corresponds to a label of a path starting at $0$).
We claim that there is a single symbol in $v$ so that if we change it to $0$,
then for the resulting word $v'$ we have that $uv'\in \G$, so $uv'$ is a label of a closed path based at the vertex $0$.
To see this, let $v \in \lang(X_\beta)$ with $|v|=k$, and let $g_1,\dots g_{k}$
be edges of $G$ whose labels give $v$. As remarked above we may assume that $g_1$ starts at the vertex $0$.
If $v=0^k$ then there is nothing to prove: we follow the closed path defining $u$ and then use $k$ times the loop based at $0$.
Otherwise, let $1\le i \le k$ be largest index so that $g_i$ is an edge labelled by a nonzero symbol.
It follows that $v=v_1\ldots v_i0^{k-i}$, where $v_i \neq 0$. Let $j$ be the initial vertex of $g_i$. Since $v_i\neq 0$ there is an edge from $j$ to $0$ labelled with $0$. That is, reading off the $i$th symbol $v_i$
of $v$, we have the option of returning to $0$. Note that $v_j=0$ for $j>i$.
Let $v'$ equal $v$ except we set $v'_i=0$. Then $uv' \in \G$
as it corresponds to the path through $G$ which starts at $0$, returns to $0$ at step $i$, and then loops at $0$ until the end of the word. Since we made only one change in $v$ to obtain $v'$ it is easy to see that $d_H(v,v')$ can be arbitrarily small if $v$ is long enough.
\end{proof}

\begin{thm} \label{thm:app_beta}
If $\mu$ is a $T_\beta$-invariant Borel probability measure, then the set of $x\in [0,1]$ which generate
$\mu$ is $\bP^0_3$-complete and the set of points with irregular $\beta$-expansion, denoted $I(T_\beta)$, is $\bS^0_3$-complete.
\end{thm}
\begin{proof}  Let $\mu$ be a $T_\beta$-invariant measure on $[0,1]$. Note that $\mu([0,1))=1$, because $1$ is never a periodic point for $T_\beta$.
If $\mu(\{0\}) =0$, then reasoning as in the proof of Theorem \ref{thm:app_gls} we see that $\mu(\Omega_\beta)=1$, and $\mu=\psi_\beta^*(\nu)$ for some shift-invariant Borel probability measure supported on $Z_\beta\subseteq Y_\beta$. Note also that $\psi_\beta(0^\infty)=0$, thus the $T_\beta$-invariant Borel probability measure concentrated at $0$ is the image through $\psi_\beta^*$ of the shift-invariant Borel probability measure concentrated at $0^\infty\in Y_\beta$. It follows that every $T_\beta$-invariant Borel probability measure on $[0,1]$ is the image through $\psi_\beta^*$ of a shift-invariant measure on $Y_\beta$.

Fix a $T_\beta$-invariant measure $\mu$ on $[0,1]$. By the above there is a shift-invariant measure $\nu$ on $Y_\beta$ such $\mu=\psi_\beta^*(\nu)$. Observe that $\nu(Y_\beta)=1$ implies that $\nu(\{\vec{e}\})=0$. Given $a_1\ldots a_k\in\lang(X_\beta)$, define the basic interval
\[
\Delta(a_1,\ldots, a_k)=\{x\in [0,1]:\iota_\beta(x)\in[a_1\ldots a_k]\}.
\]
Since $\mu=\psi_\beta^*(\nu)$, we have \[\mu(\Delta(a_1,\ldots,a_k))=\nu(\psi_\beta^{-1}(\Delta(a_1,\ldots,a_k)))\]
and
\[ \psi_\beta^{-1}(\Delta(a_1,\ldots,a_k))=
\bigg([a_1\ldots a_k]\cap Y_\beta\bigg)\cup\bigg([a_1\ldots a_k]\cap \bigcup_{n\ge 0}\sigma^{-n}(\{\vec{e}\})\bigg).\]
It follows that
\begin{equation}\label{eqn:basic_int}
\mu(\Delta(a_1,\ldots,a_k))=\nu([a_1\ldots a_k]\cap Y_\beta)=\nu([a_1\ldots a_k]).
\end{equation}

Note also that for every $T_\beta$-invariant Borel probability measure $\mu$ and a basic interval $\Delta(a_1,\ldots,a_k)$ 
we have that
\[
\partial \Delta(a_1,\ldots,a_k)\subseteq \bigcup_{n\ge 1}T^{-n}_\beta(\{0\})\cup\{1\}.
\]
(Remember that $0$ is an interior point of any interval $[0,r)$, where $r>0$.) Since basic intervals generate the Borel sigma algebra, we see that a point $x\in[0,1)$ generates $\mu$ if and only if for every $a_1\ldots a_k\in\lang(X_\beta)$ the $T_\beta$-orbit of $x$ visits the basic interval $\Delta(a_1,\ldots,a_k)$ with frequency $\mu(\Delta(a_1,\ldots,a_k))$.
Observe that if $a_1\ldots a_k\in\lang(X_\beta)$, $z\in Y_\beta$, and  $\sigma^n(z)\in[a_1\ldots a_k]\subseteq X_\beta$, then $T^n_\beta(\psi_\beta(z))$ belongs to the basic interval $\Delta(a_1,\ldots, a_k)$, since we have $\psi_\beta\circ \sigma= T_\beta\circ \psi_\beta$ on $Y_\beta$. In particular, if $z\in Y_\beta$ is generic for $\nu$, then using \eqref{eqn:basic_int} we see that $\psi_\beta(z)$ visits a basic interval $\Delta(a_1,\ldots,a_k)$ with frequency $\nu([a_1\ldots a_k])$, so $\psi_\beta(z)\in [0,1]$ generates $\mu$.
Conversely, if $x$ generates $\mu$, then the $T_\beta$-orbit of $x$ visits each basic interval $\Delta(a_1,\ldots,a_k)$ with frequency $\mu(\Delta(a_1,\ldots,a_k))$, which means that the orbit of $\iota_\beta(x)$ under $\sigma$ visits the cylinder set $[a_1\ldots a_k]$ with the same frequency, so $\iota_\beta(x)$ is generic for $\nu$ on $Y_\beta$. It follows that $\psi_\beta(G_\nu\cap Y_\beta)$ is the set of points in $[0,1)$ that generate $\mu$.

By Lemma \ref{lem:rfs-for-beta} the subshift $X_\beta$ has the
right feeble specification property. It is also known that the set of shift-invariant Borel probability measures supported on $X_\beta$ is uncountable.
Thus, $X_\beta$ satisfies the assumptions of Theorem~\ref{thm:fmt}, and we conclude that for each shift-invariant Borel probability measure $\nu$ supported on $X_\beta$ the set $G_\nu\subseteq X_\beta$ of generic points for $\nu$ is $\bP^0_3$-complete. Since $G_\nu\setminus (G_\nu\cap Y_\beta)$ is at most countable, we see that $G_\nu\cap Y_\beta$ is also a $\bP^0_3$-complete set, and $\psi_\beta$  reduces it to the set of points in $[0,1)$ that generate $\mu$. Thus the latter set is also $\bP^0_3$-complete.




Let $I(X_\beta)$ be the set of irregular points for $X_\beta$. Using Theorem \ref{thm:fmt} again, we see that $I(X_\beta)$ is a $\bs^0_3$-complete subset of $X_\beta$. Then $I(X_\beta)=(I(X_\beta)\cap X_\beta\setminus Z_\beta)\cup (I(X_\beta)\cap Z_\beta)$ is a disjoint union and $(I(X_\beta)\cap X_\beta\setminus Z_\beta)$ is at most countable. Hence $(I(X_\beta)\cap X_\beta\setminus Z_\beta)$ is a $\bP^0_3$-set. If $I(X_\beta)\cap Z_\beta$ were also a $\bP^0_3$-set, then $I(X_\beta)$ would not be $\bS^0_3$-complete, which is  absurd. Thus $I(X_\beta)\cap Z_\beta$ is a $\bS^0_3$-complete set. Reasoning as above we also obtain that $\psi_\beta(I(X_\beta)\cap Z_\beta)=I(T_\beta)\setminus\{T_\beta^n(1):n\ge 0\}$, which implies that $I(T_\beta)$ is a $\bS^0_3$-complete set.
\end{proof}

For each $\beta>1$ there is a Borel probability measure on $[0,1)$ which is invariant for the
transformation $T_\beta$, which is known as {\em Parry measure}. It is characterised as the unique
ergodic $T_\beta$-invariant that is equivalent to Lebesgue measure on $[0,1)$.
We let $\mu_\beta$ denote the Parry measure on $[0,1)$.
A real number $x \in [0,1)$ is \emph{normal with respect to the $\beta$-expansion} if $x$ generates $\mu_\beta$.

\begin{cor} \label{cor:app_beta}
For each $\beta>1$ the set of $x \in [0,1)$ which are normal with respect to the $\beta$-expansion
is a $\bP^0_3$-complete set.
\end{cor}

\subsection{Continued fraction expansions}
Our next application of Theorem~\ref{thm:fmt} is to the regular continued fraction expansion. Let
$T\colon [0,1]\setminus\Q\to [0,1]\setminus\Q$ be the continued fraction map given by
$T(x)=\frac{1}{x}-\lfloor \frac{1}{x}\rfloor$. Let $\mu$ be the Gauss measure on
$[0,1]$, which is defined by $\mu(A)= \frac{1}{\ln(2)} \int_0^1 \frac{\chi_A(x)}{1+x} \, dx$.
The Gauss measure is a $T$-invariant ergodic measure equivalent to the Lebesgue measure.
If we let $d(x)=\lfloor \frac{1}{x}\rfloor$, then the regular continued fraction expansion of $x$
is given by $d_1d_2\dots\in\N^\N$, where $d_i=d(T^{i-1}(x))$ for $i\in \N$. This expansion gives a homeomorphism $\iota$ between
$[0,1]\setminus\Q$ and $\N^\N$ such that $\iota\circ T=\sigma\circ \iota$, where $\sigma$ is the shift map on $\N^\N$.
Every $T$-invariant Borel probability measure $\mu$ on $[0,1]\setminus\Q$ corresponds to a unique shift-invariant Borel probability measure $\nu=\iota^*(\mu)$ on $\N^\N$.
Since the full shift $\N^\N$ satisfies the assumptions of Theorem \ref{thm:fmt} we see that the set of sequences generic for $\nu$, denoted by $G_\nu$, is a $\bP^0_3$-complete subset of $\N^\N$. It follows that the set of points that generate $\mu$ given by $\iota^{-1}(G_\nu)$ is also $\bP^0_3$-complete.
The same reasoning shows that the set of continued fraction irregular points is $\bS^0_3$-complete.

\begin{thm} \label{thm:app_cont}
If $\mu$ is a Borel probability measure on $[0,1]\setminus\Q$ which is invariant for the continued fraction map, then the set of points in $[0,1]\setminus\Q$ that generate $\mu$ is a $\bP^0_3$-complete set and the set of points with irregular continued fraction expansion is $\bS^0_3$-complete.
\end{thm}

\begin{cor} \label{cor:app_cont}
The set of $x \in [0,1]\setminus\Q$ which are continued fraction normal
is a $\bP^0_3$-complete set.
\end{cor}

\subsection{Sharkovsky-Sivak problem and the tent map}
Our last application considers the tent map $T \colon [0,1]\to [0,1]$ given by
\[
T(x)=\begin{cases} 2x, & \text{ if } 0 \leq x \leq \frac{1}{2}, \text{ and}\\
2-2x, & \text{ if } \frac{1}{2}< x \leq 1.\end{cases}
\]
Taking $\cI=\{I_0=[0,1/2),I_1=[1/2,1)\}$ and the function $\epsilon\colon\{0,1\}\to\{0,1\}$ such that $\epsilon(0)=0$ and $\epsilon(1)=1$ we can easily see that $T_\Ieps$ is the tent map and the $(\Ieps)$-GLS expansion map coincides with the tent map. Moreover, it is well-known and easy to see, either directly or following the reasoning presented above for the general GLS expansions, that the tent map is a factor of the full shift system $\{0,1\}^\N$ under a factor map
$\psi_{\Ieps}$ which is onto and one-to-one except at the countable set $\bigcup_{n\ge 0} T^{-n}(1/2)$, where $\psi_\Ieps$ is two-to-one (see also
\cite{deVries}, Example E in 6.3.5, taking into account
Proposition 6.3.4 (2) therein). As a corollary we obtain the following result.

\begin{cor}If $\mu$ is a Borel probability measure invariant for the tent map $T$, then the set of points that generate $\mu$ (also known as the statistical basin for $\mu$) is a $\bP^0_3$-complete set. The set of irregular points is $\bS^0_3$-complete.
\end{cor}
In particular, the statistical basin for the Dirac mass at $0$ and the tent map is a $\bP^0_3$-complete set, which answers \cite[Problem 5]{SS}.

Also as a corollary of Theorem~\ref{thm:fmt} we can answer a question of
Sharkovsky and Sivak \cite[Problem 3]{SS}, who asked whether there is a continuous map
$f \colon [0,1]\to [0,1]$ which has an invariant Borel probability measure $\mu$ such that
the set of generic points in $\bP^0_3$-complete.







\section{Concluding remarks}
Note that there are numeration systems for which our approach does not work. For example, the Cantor series expansions
are obtained through nonautonomous dynamical systems and thus
require a separate analysis.  The most up to date and general
results on normal numbers in this context are found in \cite{AireyManceHDDifference,AireyManceNormalOrders,Mance4}, respectively.
In \cite{AireyJacksonManceComplexityCantorSeries}
descriptive complexity results similar to the ones in the present paper are obtained for
Cantor series expansions. With the results presented here,
this shows that $\bp^0_3$-completeness is another universal property that holds for
all known examples of normal numbers.

\section{Acknowledgements}
The authors would like to thank the anonymous reviewer for his/her careful, constructive
and insightful comments which helped to improve this work.
The research of SJ was supported by NSF grant 1800323.
The research of DK was supported by National Science Centre (NCN) grant
2013/08/A/ST1/00275 and his stay in Rio de Janeiro, where he started to work on these problems was supported by CAPES/Brazil grant no. 88881.064927/2014-01.
DK gratefully acknowledges the great hospitality of the Federal University of Rio de Janeiro during that stay.

\bibliographystyle{amsplain}


\providecommand{\bysame}{\leavevmode\hbox to3em{\hrulefill}\thinspace}
\providecommand{\MR}{\relax\ifhmode\unskip\space\fi MR }
\providecommand{\MRhref}[2]{%
  \href{http://www.ams.org/mathscinet-getitem?mr=#1}{#2}
}
\providecommand{\href}[2]{#2}

\end{document}